\documentclass[aap]{imsart}

\RequirePackage{amsthm,amsmath,amsfonts,amssymb}
\RequirePackage[numbers,sort&compress]{natbib}
\RequirePackage[colorlinks,citecolor=blue,urlcolor=blue]{hyperref}

\startlocaldefs

\providecommand{\U}[1]{\protect \rule{.1in}{.1in}}

\theoremstyle{plain}
\newtheorem{theorem}{Theorem}

\newtheorem{corollary}[theorem]{Corollary}

\newtheorem{definition}[theorem]{Definition}

\newtheorem{lemma}{Lemma}

\newtheorem{proposition}{Proposition}
\newtheorem{remark}[theorem]{Remark}

\endlocaldefs

\begin{document}

\begin{frontmatter}

\title{Sharp Asymptotics for the Largest Component in the Subcritical Regime of Preferential Attachment Without Vertex Growth}
\runtitle{Sharp Asymptotics for the Largest Component}

\begin{aug}
\author[A]{\fnms{Yiming}~\snm{Chen}\ead[label=e1]{ymchenmath@math.pku.edu.cn}}
\address[A]{School of Mathematical Sciences, Peking University\printead[presep={,\ }]{e1}}
\end{aug}

\begin{abstract}
We study the size of the largest component in Pittel's preferential attachment process without vertex growth. Starting from the empty graph on a fixed vertex set $[n]$, edges are added one by one with probabilities proportional to $(d_u+\alpha)(d_v+\alpha)$, where $d_u$ and $d_v$ are the current degrees of $u$ and $v$, and $\alpha>0$. Let $L_1$ denote the size of the largest component, and set $m_c:=\frac{\alpha n}{2(\alpha+1)}.$ We prove that if $m=m_c(1-\varepsilon),  \varepsilon=\varepsilon(n)\to0, \varepsilon^3 n\to\infty,$ then
\[
L_1=(1+o_p(1))\frac{2(\alpha+2)}{\alpha+1}\varepsilon^{-2}\log(\varepsilon^3 n)
\]
for every fixed $\alpha>0$. Moreover, the same asymptotic holds whenever $\alpha=\alpha(n)\to a\in(0,\infty]$. In particular, the constant $2(\alpha+2)/(\alpha+1)$ converges to the Erd\H{o}s--R\'enyi value $2$ as $\alpha\to\infty$. If $m=\left\lfloor \frac n2(1-\varepsilon)\right\rfloor$ and $\alpha\varepsilon\to\infty$, then
\[
L_1=(2+o_p(1))\varepsilon^{-2}\log(\varepsilon^3 n).
\]

The subcritical asymptotics for \(L_1\) resolve the problem left open by Janson and Warnke \cite{JansonWarnke2021Preferential}. The upper bound argument relies on the fact that, after conditioning on the degree sequence, the graph can be treated through the corresponding configuration model, the lower bound follows from tree component asymptotics and a second moment argument.

\end{abstract}



\end{frontmatter}


\section{Introduction}

The emergence of a giant component is one of the fundamental phase transitions in random graph theory. In the Erd\H{o}s--R\'enyi process, the transition occurs when the number of edges crosses $n/2$, and the behavior of the largest component is now well understood throughout the phase transition, see, for example, \cite{ErdosRenyi1959Random,ErdosRenyi1960Evolution,Bollobas2001RandomGraphs,JansonLuczakRucinski2000RandomGraphs,Aldous1997Brownian,BollobasRiordan2013PhaseTransition}.  In addition to component size asymptotics, refined structural descriptions of the emerging and supercritical giant component were obtained by Ding, Kim, Lubetzky and Peres \cite{DingKimLubetzkyPeres2011YoungGiant}, and by Ding, Lubetzky and Peres \cite{DingLubetzkyPeres2014StrictlySupercritical}.

A substantial literature has extended the Erd\H{o}s--R\'enyi
phase transition theory to a broad class of sparse random graphs with prescribed or asymptotically prescribed degrees. Molloy and Reed identified the critical parameter for the appearance of a giant component in random graphs with a prescribed asymptotic degree sequence \cite{MolloyReed1995CriticalPoint,MolloyReed1998SizeGiant}. 
A detailed asymptotic theory has since been developed for the configuration model and related sparse random graph models near criticality, 
see, among many others, \cite{JansonLuczak2009NewApproach,KangSeierstad2008CriticalPhase,HatamiMolloy2012ScalingWindow,Joseph2014CriticalDegrees,DharaHofstadLeeuwaardenSen2017CriticalWindow,vanDerHofstadJansonLuczak2019BarelySupercritical,CoulsonPerarnau2023Subcritical}. 
A more structured setting is provided by preferential attachment, where the evolution of the graph depends on the current degree sequence. In classical models, vertex growth is part of the dynamics, with new vertices arriving over time and attaching preferentially to vertices of high degree
%
%
%
\cite{Price1976Cumulative,BarabasiAlbert1999Emergence,BollobasRiordanSpencerTusnady2001DegreeSequence,Barabasi2016NetworkScience,vanDerHofstad2017RGCN1,vanDerHofstad2024RGCN2}.
Many applications, however, involve networks on essentially fixed vertex sets, which motivates preferential attachment models without vertex growth, a class of models studied from the viewpoints of combinatorial probability, graph limits, and statistical physics \cite{Pittel2010Degrees,HruzPeter2010Nongrowing,BorgsChayesLovaszSosVesztergombi2011Limits,BenNaimKrapivsky2012Popularity,Rath2012EdgeConservative,RathSzakacs2012Multigraph,Janson2018EdgeExchangeable}.

%
%
%

In this paper we study Pittel's preferential attachment model $\bigl(G_{n,m}^{\alpha}\bigr)_{m\ge 0}$
on the fixed vertex set $[n]$. The process starts from the empty graph $G_{n,0}^{\alpha}=([n],\varnothing).$ For $m\ge0$, let $G_{n,m}^{\alpha}$ denote the graph produced after \(m\) steps, and let
$d_v=d_v(G_{n,m}^{\alpha})$ denote the degree of a vertex $v\in[n]$. At each step, an edge is added between a pair of distinct nonadjacent vertices. Conditionally on $G_{n,m}^{\alpha}$, the probability that the next edge is $\{v,w\}$
is given by

\[
    \mathbb{P}\!\left(
        G_{n,m+1}^{\alpha}
        =
        G_{n,m}^{\alpha}+\{v,w\}
        \,\middle|\,
        G_{n,m}^{\alpha}
    \right)
    =
    \frac{(d_v+\alpha)(d_w+\alpha)}
    {
    \sum_{\substack{{x,y}\subseteq[n]\ x\ne y,\ {x,y}\notin E(G^\alpha_{n,m})}}
   (d_x+\alpha)(d_y+\alpha)} ,
\]
where the sum in the denominator is over all unordered pairs of distinct vertices
$\{x,y\}$ that are not edges of $G_{n,m}^{\alpha}$.



 Pittel showed that, for any fixed $\alpha>0$, the giant component phase transition occurs at $m_c=\frac{\alpha n}{2(\alpha+1)},$ and obtained asymptotic estimates for the size of the largest component near $m_c$ \cite{Pittel2010Degrees}. Janson and Warnke later proved the sharp supercritical asymptotic, both for fixed $\alpha$ and in the large-$\alpha$ regime.

\begin{theorem}[Janson--Warnke \cite{JansonWarnke2021Preferential}]\label{supercritical case JW}
 Assume that $\alpha=\alpha(n)\to a\in(0,\infty]$. Let $m_c=\frac{\alpha n}{2(\alpha+1)}$, suppose that $\varepsilon=\varepsilon(n)>0$ satisfies $\varepsilon=o(1), \varepsilon^3 n\to\infty .$ Then, for $m=\lfloor m_c(1+\varepsilon)\rfloor,$ we have
\[
L_1(G^\alpha_{n,m})
=
\frac{2}{1+2/a}\,\varepsilon n\,(1+o_p(1)),
\]
with the convention $1/\infty=0$.


In particular, if $\alpha=\alpha(n)\to\infty$ and
$\alpha\varepsilon\to\infty$, then, for $m=\left\lfloor \frac{n}{2}(1+\varepsilon)\right\rfloor,$
we have
\[
    L_1(G^\alpha_{n,m})
    =
    2\varepsilon n\,(1+o_p(1)).
\]
\end{theorem}

Janson and Warnke~\cite{JansonWarnke2021Preferential} left open the problem of obtaining the sharp asymptotics in the subcritical regime.
Their method of conditioning
on the degree sequence reduces the multigraph version of \(G^\alpha_{n,m}\) to a
configuration model whose degree sequence is asymptotically negative binomial. However, the available configuration model result in the supercritical regime gives only
giant component information. In the subcritical regime, it yields
\(L_1 = O_p(n)\), which does not detect the logarithmic scale $L_1 \asymp \varepsilon^{-2}\log(\varepsilon^3 n).$


By contrast, the subcritical lower bound is a rare
event problem in which one must prove the existence of tree components in a narrow
logarithmic window and deal with their dependence. In this paper, our main result establish the subcritical estimate, allowing the attachment parameter \(\alpha\) to depend on \(n\).


\begin{theorem}\label{subcritical case theo 2}
Assume that \(\alpha=\alpha(n)\to a\in(0,\infty]\), and set $m_c=\frac{\alpha n}{2(\alpha+1)}.$ If \(m=\lfloor m_c(1-\varepsilon)\rfloor\), \(\varepsilon=\varepsilon(n)=o(1)\), and
\(\varepsilon^3 n\to\infty\), then
\begin{equation}\label{subcritical case}
L_1=(1+o_p(1))\,\frac{2(\alpha+2)}{\alpha+1}\,\varepsilon^{-2}\log(\varepsilon^3 n).
\end{equation}
Moreover, if \(m=\lfloor n(1-\varepsilon)/2\rfloor\), \(\alpha\to\infty\), and
\(\alpha\varepsilon\to\infty\), then
\begin{equation}\label{subcritical case II}
L_1=(2+o_p(1))\varepsilon^{-2}\log(\varepsilon^3 n).
\end{equation}
\end{theorem}

\eqref{subcritical case} and \eqref{subcritical case II} resolve a problem posed in \cite{JansonWarnke2021Preferential}. Together with the supercritical Theorem \ref{supercritical case JW}, Theorem \ref{subcritical case theo 2} completes the asymptotic analysis on both sides of the phase transition. 



The strategy of the proof is as follows. For the upper bound, following \cite{JansonWarnke2021Preferential}, by conditioning on the random degree sequence of the multigraph, we prove that with high probability the degree sequence lies in a class of sequences with bounded empirical distribution, moments, and exponential moments. The upper bound then follows from the subcritical configuration model result of Coulson and Perarnau
\cite{CoulsonPerarnau2023Subcritical}. For the lower bound, we compute the first moment for the number of tree components of a given size, and then prove that after conditioning on a prescribed tree component, the induced multigraph on the remaining vertices again has the law of the same preferential attachment model with updated parameters, which makes it possible to carry out a second moment argument in a logarithmic window of sizes.




The rest of the paper is organized as follows. Section~2 introduces the notation and the basic lemmas from \cite{JansonWarnke2021Preferential,CoulsonPerarnau2023Subcritical} that will be used throughout. The proof of the main subcritical theorem is carried out in Section~3. 

\section{Preliminaries}

\subsection{Model and notation}

We first introduce a multigraph version of $G_{n,m}^{\alpha}$, whose conditional distribution given its degree sequence is the configuration model \cite{JansonWarnke2021Preferential}.
\begin{definition}
Let $\alpha>0$ and $n\in\mathbb N$. Define
$(G^{\alpha,*}_{n,m})_{m\ge0}$ to be the multigraph version of
$(G^\alpha_{n,m})_{m\ge0}$ on $[n]$, starting from the
empty graph, allowing loops and multiple edges.

\end{definition}

 Set $\ell=2m$, $\mu_n=\frac{\ell}{n}=\frac{\alpha}{\alpha+1}(1-\varepsilon),$ $p_n=\frac{\ell}{\alpha n+\ell}=\frac{1-\varepsilon}{\alpha+2-\varepsilon},$ and $x_n=\varepsilon^3 n.$ For a deterministic degree sequence
$d=(d_v)_{v\in[n]}$ with positive even total degree $\ell(d)=\sum_{v=1}^n d_v,$ let $S_d$ be the set of all sequences
$w=(w_1,\ldots,w_{\ell(d)})$ in which each vertex $v\in[n]$ appears
exactly $d_v$ times. Choose $w\in S_d$ uniformly at random, and define the
edge multiset $E_d=\{w_1w_2,w_3w_4,\ldots,w_{\ell(d)-1}w_{\ell(d)}\}.$

Let $\operatorname{CM}(d):=([n],E_d)$ denote the configuration model with degree sequence $d$, and write $\pi_k(d)=\frac1n\sum_{v=1}^n \mathbf 1_{\{d_v=k\}},$ $\mu_r(d)=\frac1n\sum_{v=1}^n d_v^r,$ and $\Delta(d)=\max_{v\in[n]} d_v.$ Let $D_d$ be the degree of a uniformly chosen vertex so that $$\mathbb P(D_{ d}=k)=\pi_k( d),$$ and let $D_d^\star$ be its size biased version defined by

$$\mathbb P(D_{ d}^*=k)=\frac{k\pi_k( d)}{\mu_1( d)}.$$


 Then we introduce the associated exploration parameters, write
\[
Q(d)=\mathbb E(D_d^\star-2)=\frac{\mu_2(d)-2\mu_1(d)}{\mu_1(d)}
\]
for the corresponding drift, and
\[
R(d)=\mathbb E(D_d^\star-2)^2=\frac{\mu_3(d)-4\mu_2(d)+4\mu_1(d)}{\mu_1(d)}
\]
for the quadratic scale of the fluctuations. Finally, for exponential
bounds we set
\[\phi_d(\theta)=\mathbb E e^{\theta(D_d^\star-2)}.\]

For the multigraph $G_{n,m}^{\alpha,\ast}$, let $C_k^{(n,m)}$ denote the number of tree components of size $k$. For $u>r-1$ and $r\in\mathbb N$, write
\[
\binom{u}{r}=\frac{\Gamma(u+1)}{\Gamma(r+1)\Gamma(u-r+1)}.
\]

\subsection{Basic facts}

Before turning to the main result, we recall several preliminaries. Proposition~\ref{prop:cp-upper} gives an upper bound for the subcritical configuration model and is an immediate corollary of \cite{CoulsonPerarnau2023Subcritical}, Theorem~1.5. Since all the required assumptions in Coulson and Perarnau~\cite{CoulsonPerarnau2023Subcritical} are satisfied in our setting, and \eqref{equ-prop} follows by taking \(\varepsilon=\delta\).

\begin{proposition}\label{prop:cp-upper}
Let $(d_n)_{n\ge 1}$ be deterministic degree sequences, and write
\(\ell_n = \ell(\mathbf d_n)\), \(D_n = D_{\mathbf d_n}\),
\(\Delta_n = \Delta(\mathbf d_n)\), and \(\phi_n = \phi_{\mathbf d_n}\).
Assume that:
\begin{enumerate}
\item $D_n \Rightarrow D$ for some random variable \(D\) taking values in \(\mathbb Z_{\ge0}\);
\item $Q(d_n)\to 0$;
\item $\mathbb P(D=0)>0$ and $\mathbb P(D=1)>0$;
\item $\mathbb P(D_n=0)>0$ and $\mathbb P(D_n=1)>0$ for all large $n$;
\item $\mu_4(d_n)\le \Delta_n^{1/2}$ and $\Delta_n\le n^{1/6}$ for all large $n$.
\end{enumerate}
Assume further that, for all large $n$, the equation $\phi_n'(\theta)=0$ has a smallest positive solution $\theta_n\in(0,1)$, and define
\begin{equation}\label{equ-def-tn}
T_n:=\frac{1}{-\log\phi_n(\theta_n)}
\left[
\log n+\frac{3}{2}\log\bigl(-\log\phi_n(\theta_n)\bigr)
-\frac12\log\phi_n''(\theta_n)
+\log\mathbb E\!\left(D_n e^{\theta_n D_n}\right)
\right].
\end{equation}
If $\theta_n\ell_n/T_n\to\infty$, then for every fixed $\delta>0$,
\begin{equation}\label{equ-prop}
\mathbb P\bigl(L_1(\mathrm{CM}(d_n))\le (1+\delta)T_n\bigr)\to 1.
\end{equation}
\end{proposition}

\begin{remark}
Condition \textup{4} is in fact eventually implied by
\textup{1} and \textup{3}. Indeed, since \(D_n \Rightarrow D\), we have
\[
    \mathbb P(D_n=i)\to \mathbb P(D=i), \qquad i=0,1.
\]
Thus \(\mathbb P(D_n=0)>0\) and \(\mathbb P(D_n=1)>0\) for all sufficiently
large \(n\). We nevertheless keep Condition \textup{4}, since it
directly ensures lattice span 1 for \(D_n\), and hence includes the lattice
condition required in Coulson--Perarnau~\cite{CoulsonPerarnau2023Subcritical}.
\end{remark}

The remaining results from Janson and Warnke \cite{JansonWarnke2021Preferential} relate the preferential attachment process to its multigraph and configuration model representations.


\begin{lemma}[{\cite{JansonWarnke2021Preferential}, Theorem~2.5, Remark~2.6, and Remark~3.2}]\label{lem:jw-degree-law}
Let $Y_1,\dots,Y_n$ be i.i.d.\ random variables with
\[
\mathbb P(Y_i=k)=\binom{\alpha+k-1}{k}(1-p_n)^{\alpha}p_n^k,
\]
where $$p_n=\frac{2m}{\alpha n+2m}=\frac{1-\varepsilon}{\alpha+2-\varepsilon}.$$
Then
\[
d\bigl(G_{n,m}^{\alpha,\ast}\bigr)\overset{d}= (Y_1,\dots,Y_n)\,\big|\,\{Y_1+\cdots+Y_n=2m\}.
\]
In particular, if $Y^{(n)}\sim\operatorname{NBin}(\alpha,p_n)$, then for every fixed integer $k\ge 0$ and every fixed integer $r\ge 1$,
\[
\pi_k(d)=\mathbb P\bigl(Y^{(n)}=k\bigr)+O_p(n^{-1/2}),
\]
and
$$\mu_r(d)=\mathbb E\bigl[(Y^{(n)})^r\bigr]+O_p(n^{-1/2}).$$
For each fixed $r\ge 1$ one also has $\mathbb E\bigl[(Y^{(n)})^r\bigr]=O(1)$.
\end{lemma}

\begin{lemma}[{\cite{JansonWarnke2021Preferential}, Theorem~2.4}]\label{lem:jw-cm}
For any degree sequence $d=(d_v)_{v\in[n]}$ satisfying $\sum_{v\in[n]} d_v=2m,$ the conditional law of $G_{n,m}^{\alpha,\ast}$ given $d\bigl(G_{n,m}^{\alpha,\ast}\bigr)=d$ is $\mathrm{CM}(d)$. Equivalently, for every family $\mathcal H_n$ of multigraphs with vertex set $[n]$ and $m$ edges,
\[
\mathbb P\bigl(G_{n,m}^{\alpha,\ast}\in \mathcal H_n \,\big|\, d(G_{n,m}^{\alpha,\ast})=d\bigr)
=
\mathbb P\bigl(\mathrm{CM}(d)\in \mathcal H_n\bigr).
\]
\end{lemma}

\begin{lemma}[{\cite{JansonWarnke2021Preferential}, Theorem~2.2}]\label{lem:jw-transfer}
Given $C,\alpha_0>0$, there exists $B=B(C,\alpha_0)>0$ such that the following holds whenever $1\le m\le Cn$ and $\alpha\ge \alpha_0$. For every family $\mathcal G_n$ of simple graphs with vertex set $[n]$ and exactly $m$ edges,
\[
\mathbb P\bigl(G_{n,m}^{\alpha}\in \mathcal G_n\bigr)
\le
B\,\mathbb P\bigl(G_{n,m}^{\alpha,\ast}\in \mathcal G_n\bigr)+o(1).
\]
\end{lemma}

\section{Main Result and Proof}\label{sec:subcritical}

\begin{theorem}\label{thm:main-fixed-alpha}
Assume that $\alpha=\alpha(n)\to a\in(0,\infty]$, and let $m_c$ and $m$ defined as above with $\varepsilon=\varepsilon(n)\downarrow 0$ and $\varepsilon^3 n\to\infty$.
For $C_\alpha=\frac{2(\alpha+2)}{\alpha+1}.$ Then
\[
L_1(G_{n,m}^{\alpha,\ast})=(1+o_p(1))C_\alpha\,\varepsilon^{-2}\log(\varepsilon^3 n),
\]
and
\[
L_1(G_{n,m}^{\alpha})=(1+o_p(1))C_\alpha\,\varepsilon^{-2}\log(\varepsilon^3 n).
\]
\end{theorem}

\subsection{The upper bound}

Coulson and Perarnau \cite{CoulsonPerarnau2023Subcritical} established upper bounds for the largest component of subcritical configuration models with prescribed degree sequences. Under mild regularity conditions, their bounds are asymptotically sharp for given degree sequences and improve the earlier barely subcritical estimates of Hatami and Molloy \cite{HatamiMolloy2012ScalingWindow}. Their result supplies our upper bound after conditioning on the random degree sequence.

We first construct a deterministic set \(\mathcal D_n\) of admissible degree sequences such that the degree sequence of \(G^{\alpha,*}_{n,m}\) belongs to \(\mathcal D_n\) with high probability. Let $Y^{(n)}\sim\operatorname{NBin}(\alpha,p_n)$, where $p_n=\frac{1-\varepsilon}{\alpha+2-\varepsilon}.$ Define the limiting law
\[
Y\sim
\begin{cases}
\operatorname{NBin}\!\left(a,\dfrac{1}{a+2}\right), & a<\infty,\\[1ex]
\operatorname{Poisson}(1), & a=\infty.
\end{cases}
\]
Since $\alpha\to a\in(0,\infty]$, there exists $\alpha_*>0$ such that $\alpha\ge \alpha_*$ for all sufficiently large $n$. 

\begin{lemma}\label{lem:good-degree-unified}
There exist sequences $M_n\to\infty, \delta_n\downarrow0, \rho_n\downarrow0,$ constants $C_0<\infty$ and $\beta>0$, and sets $\mathcal D_n$ of degree sequences such that
\[
\mathbb P\bigl(d(G_{n,m}^{\alpha,\ast})\in\mathcal D_n\bigr)=1-o(1).
\]
Moreover, every $d_n\in\mathcal D_n$ satisfies 
\begin{align}
\max_{0\le k\le M_n}\Bigl|\pi_k(d_n)-\mathbb P(Y^{(n)}=k)\Bigr|&\le \delta_n,\label{eq:good-pi}\\
\Bigl|\mu_j(d_n)-\mathbb E\bigl[(Y^{(n)})^j\bigr]\Bigr|&\le \rho_n\varepsilon \qquad (j=1,2,3),\label{eq:good-mu123}\\
\mu_4(d_n)&\le C_0,\label{eq:good-mu4}\\
\frac1n\sum_{v=1}^n d_{n,v}^r e^{\beta d_{n,v}}&\le C_0 \qquad (r=1,2,3,4).\label{eq:good-exp}
\end{align}
Each such $d_n\in\mathcal D_n$ satisfies assumptions \textup{(1)}--\textup{(5)}
of Proposition~\ref{prop:cp-upper} with limit law $D=Y$, and
\[
Q(d_n)=-\varepsilon+o(\varepsilon),
\]
$$R(d_n)=\frac{\alpha+2}{\alpha+1}+o(1).$$
\end{lemma}

\begin{proof}
Write $d=d(G_{n,m}^{\alpha,\ast}).$ By Lemma~\ref{lem:jw-degree-law}, for every fixed integer $k\ge 0$ and every fixed integer $j\ge1$, we have
\[
\pi_k(d)=\mathbb P(Y^{(n)}=k)+O_p(n^{-1/2}),
\]
and $$\mu_j(d)=\mathbb E\bigl[(Y^{(n)})^j\bigr]+O_p(n^{-1/2}).$$
Since $x_n=\varepsilon^3 n\to\infty$, we have $\varepsilon\gg n^{-1/3}$ and hence $n^{-1/2}=o(\varepsilon).$ Therefore, for $j=1,2,3$,
\[
\mu_j(d)=\mathbb E\bigl[(Y^{(n)})^j\bigr]+o_p(\varepsilon).
\]
Also, for $j=4$, Lemma~\ref{lem:jw-degree-law} gives
\[
\mu_4(d)=\mathbb E\bigl[(Y^{(n)})^4\bigr]+O_p(n^{-1/2}).
\]

We claim that $\mathbb E\bigl[(Y^{(n)})^4\bigr]=O(1).$ Indeed, for every fixed integer $r\ge1$, the factorial moment formula for the negative binomial law gives
\begin{equation}
\label{eq:factorial-moment-Yn}
\mathbb E\bigl[(Y^{(n)})_r\bigr]
=
\alpha(\alpha+1)\cdots(\alpha+r-1)
\left(\frac{p_n}{1-p_n}\right)^r
=
\alpha(\alpha+1)\cdots(\alpha+r-1)
\left(\frac{1-\varepsilon}{\alpha+1}\right)^r .
\end{equation}
For each fixed $r$, \eqref{eq:factorial-moment-Yn} is $O(1)$ uniformly in $n$. Since ordinary moments are linear
combinations of factorial moments, it follows that
\[
\mathbb E\bigl[(Y^{(n)})^r\bigr]=O(1)
\qquad\text{for each fixed }r\ge1.
\]
In particular,
\[
\mu_4(d)=O_p(1).
\]

We next choose the sequences \(M_n\), \(\delta_n\), and \(\rho_n\) so that \eqref{eq:good-pi}--\eqref{eq:good-mu123} hold with high probability.
For every fixed $M\ge0$,
\[
\max_{0\le k\le M}\Bigl|\pi_k(d)-\mathbb P(Y^{(n)}=k)\Bigr|\xrightarrow{p}0.
\]
By a diagonal selection, there exist $M_n\to\infty$ and
$\delta_n\downarrow0$ such that
\[
\mathbb P\!\left(
\max_{0\le k\le M_n}\Bigl|\pi_k(d)-\mathbb P(Y^{(n)}=k)\Bigr|\le \delta_n
\right)=1-o(1).
\]
Likewise,
\[
\max_{1\le j\le 3}\frac{\bigl|\mu_j(d)-\mathbb E[(Y^{(n)})^j]\bigr|}{\varepsilon}\xrightarrow{p}0,
\]
so there exists a deterministic sequence $\rho_n\downarrow0$ such that
\[
\mathbb P\!\left(
\bigl|\mu_j(d)-\mathbb E[(Y^{(n)})^j]\bigr|\le \rho_n\varepsilon
\text{ for }j=1,2,3
\right)=1-o(1).
\]
After enlarging $C_0$ if necessary, we may also assume that
\[
\mathbb P\bigl(\mu_4(d)\le C_0\bigr)=1-o(1).
\]

It remains to prove \eqref{eq:good-exp}. We use the continuous time embedding from the proof of \cite{JansonWarnke2021Preferential}, Theorem~2.5: there exist independent pure birth processes
$(D_v(t))_{t\ge0}$, $v\in[n]$, and stopping times $(\tau_j)_{j\ge0}$ such that $d_v=D_v(\tau_{2m}).$ Set $t_n=\log\Bigl(1+\frac{2}{\alpha}\Bigr)$. Then the one dimensional marginals of the embedding satisfy
\[
D_v(t_n)\sim \operatorname{NBin}\!\left(\alpha,\frac{2}{\alpha+2}\right),
\]
independently over $v\in[n]$. Hence
\[
\mathbb E D_v(t_n)=2,
\qquad
\operatorname{Var}(D_v(t_n))=2+\frac{4}{\alpha}=O(1).
\]
Therefore
\[
\mathbb E\sum_{v=1}^n D_v(t_n)=2n,
\]
and $$\operatorname{Var}\!\left(\sum_{v=1}^n D_v(t_n)\right)=O(n).$$
On the other hand, $2m=\frac{\alpha}{\alpha+1}(1-\varepsilon)n\le n.$ Consequently,
\[
\mathbb E\!\left[\sum_{v=1}^n D_v(t_n)-2m\right]\ge n.
\]
Chebyshev's inequality yields
\[
\mathbb P(\tau_{2m}>t_n)
=
\mathbb P\!\left(\sum_{v=1}^n D_v(t_n)<2m\right)
=o(1).
\]

Choose $\beta>0$ so small that $\frac{2e^{4\beta}}{\alpha_*+2}<1.$ Then, for all sufficiently large $n$,
\[
\frac{2e^{4\beta}}{\alpha+2}\le \frac{2e^{4\beta}}{\alpha_*+2}<1,
\]
and hence the moment generating function of $D_v(t_n)$ is finite at $4\beta$. More precisely,
\begin{equation}\label{eq:mgf-Dv-tn}
\mathbb E e^{4\beta D_v(t_n)}=
\left(
\frac{\alpha/(\alpha+2)}{1-\frac{2e^{4\beta}}{\alpha+2}}
\right)^\alpha=
\left(
\frac{\alpha}{\alpha+2-2e^{4\beta}}
\right)^\alpha.
\end{equation}
Since $2e^{4\beta}<\alpha_*+2\le \alpha+2$, the base on the right hand side of \eqref{eq:mgf-Dv-tn} equals $\left(1-\frac{2e^{4\beta}-2}{\alpha}\right)^{-1}.$ Hence it is bounded uniformly in $n$, indeed,
\[
\left(1-\frac{2e^{4\beta}-2}{\alpha}\right)^{-\alpha}
\le
\sup_{x\ge \alpha_*}\left(1-\frac{2e^{4\beta}-2}{x}\right)^{-x}
<\infty.
\]
Thus
\[
\sup_n \mathbb E e^{4\beta D_v(t_n)}<\infty.
\]

For each $r\in\{1,2,3,4\}$ and $z\ge0$, there exists $C_r<\infty$ such that $z^{2r}e^{2\beta z}\le C_r e^{4\beta z}$. Hence
\[
\sup_n \mathbb E\bigl[D_v(t_n)^{2r} e^{2\beta D_v(t_n)}\bigr]<\infty. 
\]
Since the random variables $D_v(t_n)^r e^{\beta D_v(t_n)}$ are independent over $v$, it follows that
\[
\operatorname{Var}\!\left(\frac1n\sum_{v=1}^n D_v(t_n)^r e^{\beta D_v(t_n)}\right)=O(n^{-1}).
\]
By Chebyshev's inequality, after enlarging $C_0$ if necessary,
\[
\mathbb P\!\left(
\frac1n\sum_{v=1}^n D_v(t_n)^r e^{\beta D_v(t_n)}\le C_0
\right)=1-o(1).
\]
On the event $\{\tau_{2m}\le t_n\}$ we have $d_v=D_v(\tau_{2m})\le D_v(t_n)$ for every $v$, and hence
\[
\frac1n\sum_{v=1}^n d_v^r e^{\beta d_v}
\le
\frac1n\sum_{v=1}^n D_v(t_n)^r e^{\beta D_v(t_n)}.
\]
Together with $\mathbb P(\tau_{2m}>t_n)=o(1)$ this proves \eqref{eq:good-exp} with probability $1-o(1)$.

Let $\mathcal D_n$ be the set of all degree sequences satisfying \eqref{eq:good-pi}-\eqref{eq:good-exp}. Then
\[
\mathbb P(d\in\mathcal D_n)=1-o(1).
\]

Now fix a sequence $d_n\in\mathcal D_n$, and let $D_{d_n}$ denote the degree of a uniformly
chosen vertex of $d_n$. Thus
\[
\mathbb P(D_{d_n}=k)=\frac{1}{n}\#\{v : d_{n,v} = k\}=\pi_k(d_n).
\]
We now prove that $D_{d_n}\Rightarrow Y$. If $a<\infty$, then $\alpha\to a$ and $p_n\to 1/(a+2)$, so for every fixed $k\ge0$, the negative binomial masses converge pointwise,
\[
\mathbb P(Y^{(n)}=k)\to \mathbb P(Y=k).
\]
If $a=\infty$, then $p_n\to0$ and
\[
\alpha p_n=\frac{\alpha(1-\varepsilon)}{\alpha+2-\varepsilon}\to1.
\]
For every fixed $k\ge0$,
\[
\mathbb P(Y^{(n)}=k)
=
\frac1{k!}\prod_{j=0}^{k-1}(\alpha+j)p_n^k (1-p_n)^\alpha
=
\frac1{k!}\prod_{j=0}^{k-1}(\alpha p_n+jp_n)\,(1-p_n)^\alpha
\to \frac{e^{-1}}{k!}=\mathbb P(Y=k),
\]
so again $Y^{(n)}\Rightarrow Y$.

Let $A\ge1$ be fixed. For all large $n$ we have $A\le M_n$, and therefore
\[
\sum_{k=0}^A \bigl|\mathbb P(D_{d_n}=k)-\mathbb P(Y=k)\bigr|
\le
(A+1)\delta_n+\sum_{k=0}^A \bigl|\mathbb P(Y^{(n)}=k)-\mathbb P(Y=k)\bigr|
\to0.
\]
Also, by \eqref{eq:good-mu4},
\[
\mathbb P(D_{d_n}>A)\le \frac{\mu_4(d_n)}{A^4}\le \frac{C_0}{A^4},
\]
and similarly $\mathbb P(Y>A)\le C/A^4$ for some constant $C<\infty$ based on the fact that $Y$ has finite fourth moment.
Hence
\[
\sum_{k>A}\bigl|\mathbb P(D_{d_n}=k)-\mathbb P(Y=k)\bigr|
\le \mathbb P(D_{d_n}>A)+\mathbb P(Y>A)
\le \frac{C_0+C}{A^4}.
\]
Letting first $n\to\infty$ and then $A\to\infty$ proves $D_{d_n}\Rightarrow Y$.

Since $Y$ is either $\operatorname{NBin}(a,1/(a+2))$ with $a<\infty$ or $\operatorname{Poisson}(1)$, we have
\[
\mathbb P(Y=0)>0,
\qquad
\mathbb P(Y=1)>0.
\]
Fix an integer $K\ge0$. Because $\mathbb P(Y^{(n)}=K)\to \mathbb P(Y=K)>0$, property
\eqref{eq:good-pi} implies that for all large $n$, $\pi_K(d_n)>0.$ Taking $K=0$ and $K=1$ gives assumptions \textup{(3)} and \textup{(4)} of Proposition~\ref{prop:cp-upper}. As \(K\) was arbitrary, it follows that $\Delta(d_n)\to\infty.$ 

On the other hand, \eqref{eq:good-exp} with $r=1$ gives
\[
\Delta(d_n)e^{\beta\Delta(d_n)}
\le \sum_{v=1}^n d_{n,v} e^{\beta d_{n,v}}
\le C_0 n,
\]
so $\Delta(d_n)=O(\log n).$ Consequently, for all sufficiently large $n$, $\Delta(d_n)\le n^{1/6},$ and
\[
\mu_4(d_n)\le C_0\le \Delta(d_n)^{1/2}.
\]
Thus every $d_n\in\mathcal D_n$ satisfies assumptions \textup{(1)}--\textup{(5)} of Proposition~\ref{prop:cp-upper}.

Finally we compute $Q(d_n)$ and $R(d_n)$. Let $(Y^{(n)})^\star$ be the size-biased version of $Y^{(n)}$.
A direct calculation shows that $(Y^{(n)})^\star-1\sim \operatorname{NBin}(\alpha+1,p_n).$ Hence
\[
\mathbb E\bigl[(Y^{(n)})^\star-2\bigr]
=
\frac{(\alpha+1)p_n}{1-p_n}-1
=
1-\varepsilon-1
=
-\varepsilon,
\]
and
\[
\mathbb E\bigl[((Y^{(n)})^\star-2)^2\bigr]
=
\frac{(\alpha+1)p_n}{(1-p_n)^2}+\varepsilon^2
=
\frac{(1-\varepsilon)(\alpha+2-\varepsilon)}{\alpha+1}+\varepsilon^2
=
\frac{\alpha+2}{\alpha+1}+O(\varepsilon).
\]
Using
\[
Q(d)=\frac{\mu_2(d)-2\mu_1(d)}{\mu_1(d)},\]
and\[
R(d)=\frac{\mu_3(d)-4\mu_2(d)+4\mu_1(d)}{\mu_1(d)},
\]
together with \eqref{eq:good-mu123} implies
\[
\mu_j(d_n)=\mathbb E[(Y^{(n)})^j]+o(\varepsilon)\qquad (j=1,2,3),
\]
while
\[
\mu_1(d_n)=\frac{2m}{n}=\frac{\alpha}{\alpha+1}(1-\varepsilon)
\]
is bounded away from $0$, it follows that
\[
Q(d_n)=-\varepsilon+o(\varepsilon),
\]
\[
R(d_n)=\frac{\alpha+2}{\alpha+1}+o(1).
\]
\end{proof}

Using the negative drift and a second moment from Lemma~\ref{lem:good-degree-unified}, we show that
$T_n$ in Proposition~\ref{prop:cp-upper} determines the scale of the largest component in the subcritical regime.
The idea is to use the exponential moment bound to obtain Taylor expansions for
$\phi_n(\theta)$, then solve the saddle point equation.

\begin{lemma}\label{lem:saddle-unified}
For every $d_n\in\mathcal D_n$, let $\eta_n=D^\star_{d_n}-2,$ then $Q_n=\mathbb E\eta_n,$ $R_n=\mathbb E\eta_n^2,$ Let \(\phi_n=\phi_{d_n}\) and \(\theta_n\) be as in
Proposition \ref{prop:cp-upper}. Then
\[
\theta_n=\frac{\alpha+1}{\alpha+2}\varepsilon+o(\varepsilon),
\]
and
\[
\phi_n''(\theta_n)=\frac{\alpha+2}{\alpha+1}+o(1).
\]
For \(T_n\) defined in \eqref{equ-def-tn}, we have
\[
T_n=(1+o(1))\,C_\alpha\,\varepsilon^{-2}\log(\varepsilon^3 n),
\qquad
\frac{\theta_n\ell(d_n)}{T_n}\to\infty.
\]
\end{lemma}

\begin{proof}
By Lemma~\ref{lem:good-degree-unified}, $Q_n=-\varepsilon+o(\varepsilon),$ and $R_n=\frac{\alpha+2}{\alpha+1}+o(1).$ Moreover, since $\mu_1(d_n)\ge \frac{\alpha_*}{2(\alpha_*+1)}$ for all large $n$ and
\eqref{eq:good-exp} holds, there exists a constant $C<\infty$ depending only on $\beta$ such that
\[
d_{n,v}|d_{n,v}-2|^3 e^{\beta|d_{n,v}-2|}
\le C\bigl(d_{n,v}^4 e^{\beta d_{n,v}}+d_{n,v}e^{2\beta}\bigr)
\qquad (v\in[n]).
\]
Hence
\[
\mathbb E\bigl[|\eta_n|^3 e^{\beta|\eta_n|}\bigr]
=
\frac{1}{\mu_1(d_n)}\cdot \frac1n \sum_{v=1}^n
d_{n,v}|d_{n,v}-2|^3 e^{\beta|d_{n,v}-2|}
=O(1).
\]
Therefore Taylor expansion with remainder gives, uniformly for $0\le \theta\le \beta/2$,
\begin{align}
\phi_n(\theta)&=1+Q_n\theta+\frac{R_n}{2}\theta^2+O(\theta^3),\label{eq:taylor-phi}\\
\phi_n'(\theta)&=Q_n+R_n\theta+O(\theta^2),\label{eq:taylor-phip}\\
\phi_n''(\theta)&=R_n+O(\theta).\label{eq:taylor-phipp}
\end{align}
Here the implied constants are uniform over $n$ and $\theta$.

For $\theta\in\mathbb R$, since $\phi_n''(\theta)=\mathbb E\bigl[\eta_n^2 e^{\theta\eta_n}\bigr]>0$, the derivative $\phi_n'$ is strictly increasing. For all large $n$, $Q_n<0$, and \eqref{eq:taylor-phip} yields
\[
\phi_n'\!\left(\frac{2|Q_n|}{R_n}\right)
=
Q_n+2|Q_n|+O(Q_n^2)
=
|Q_n|+O(Q_n^2)
>0
\]
for all large $n$. Therefore there exists a unique positive zero of $\phi_n'$, and it is the smallest positive
solution $\theta_n$. In particular,
\[
0<\theta_n\le \frac{2|Q_n|}{R_n}=O(\varepsilon),
\]
so $\theta_n\le \beta/2$ for all large $n$.

Substituting $\theta=\theta_n$ into \eqref{eq:taylor-phip} and using $\phi_n'(\theta_n)=0$, we obtain
\[
0=Q_n+R_n\theta_n+O(\theta_n^2)
=Q_n+R_n\theta_n+O(Q_n^2).
\]
Hence
\[
\theta_n=-\frac{Q_n}{R_n}+O(Q_n^2)
=
\frac{\alpha+1}{\alpha+2}\varepsilon+o(\varepsilon).
\]

Next, \eqref{eq:taylor-phi} and $\theta_n=O(\varepsilon)$ imply
\[
\phi_n(\theta_n)
=
1+Q_n\theta_n+\frac{R_n}{2}\theta_n^2+O(\varepsilon^3)
=
1-\frac{Q_n^2}{2R_n}+O(\varepsilon^3).
\]
Since $Q_n^2/R_n=O(\varepsilon^2)$, taking logarithms yields
\[
\log \phi_n(\theta_n)
=
-\frac{Q_n^2}{2R_n}+O(\varepsilon^3),
\]
and therefore
\[
-\log \phi_n(\theta_n)
=
\frac{Q_n^2}{2R_n}+O(\varepsilon^3)
=
\frac{\alpha+1}{2(\alpha+2)}\varepsilon^2+o(\varepsilon^2).
\]
From \eqref{eq:taylor-phipp} we also obtain
\[
\phi_n''(\theta_n)=R_n+O(\theta_n)=\frac{\alpha+2}{\alpha+1}+o(1).
\]

Furthermore, \eqref{eq:good-exp} with $r=1$ and $\theta_n\le \beta/2$ imply
\begin{equation}\label{eq:Ddn-exponential-moment-bound}
\begin{aligned}
\mathbb E\bigl[D_{d_n}e^{\theta_n D_{d_n}}\bigr]\le
\mathbb E\bigl[D_{d_n}e^{\beta D_{d_n}/2}\bigr] \le
\mathbb E\bigl[D_{d_n}e^{\beta D_{d_n}}\bigr] =
\frac{1}{n}\sum_{v=1}^n d_{n,v} e^{\beta d_{n,v}} \le C_0.
\end{aligned}
\end{equation}

Using \eqref{eq:Ddn-exponential-moment-bound} in the definition of $T_n$ in Proposition~\ref{prop:cp-upper}, we obtain
\[
T_n
=
\frac{1}{-\log \phi_n(\theta_n)}
\left[
\log n+\frac32\log\!\bigl(-\log\phi_n(\theta_n)\bigr)
-\frac12\log \phi_n''(\theta_n)
+\log \mathbb E\bigl[D_{d_n}e^{\theta_n D_{d_n}}\bigr]
\right].
\]
Now
\[
-\log \phi_n(\theta_n)
=
\frac{\alpha+1}{2(\alpha+2)}\varepsilon^2(1+o(1)),
\]
so
\[
\log\!\bigl(-\log \phi_n(\theta_n)\bigr)
=
2\log\varepsilon+O(1).
\]

Also note that
\[
\frac{\alpha+1}{2(\alpha+2)}\in \Bigl[\frac{\alpha_*+1}{2(\alpha_*+2)},\,\frac12\Bigr].
\]
Therefore
\[
T_n
=
C_\alpha \varepsilon^{-2}(1+o(1))\bigl(\log n+3\log\varepsilon+O(1)\bigr).
\]
Since $\log(\varepsilon^3 n)=\log n+3\log\varepsilon\to\infty,$ it follows that
\[
T_n=(1+o(1))\,C_\alpha\,\varepsilon^{-2}\log(\varepsilon^3 n).
\]

Finally,
\[
\ell(d_n)=2m=\frac{\alpha}{\alpha+1}(1-\varepsilon)n\asymp n,
\qquad
\theta_n\asymp \varepsilon,
\]
Moreover, $C_\alpha\in[2,4]$. Hence
\[
\frac{\theta_n\ell(d_n)}{T_n}
\asymp
\frac{\varepsilon n}{\varepsilon^{-2}\log(\varepsilon^3 n)}
=
\frac{\varepsilon^3 n}{\log(\varepsilon^3 n)}
=
\frac{x_n}{\log x_n}\to\infty.
\]
\end{proof}

\begin{theorem}[Upper bound]\label{thm:upper-unified}
Fix $\delta>0$, and set $c_n=C_\alpha \varepsilon^{-2}\log(\varepsilon^3 n).$ Then
\[
\mathbb P\bigl(L_1(G_{n,m}^{\alpha})>(1+\delta)c_n\bigr)\to0.
\]
%
\end{theorem}

\begin{proof}
We first prove the multigraph bound. Suppose, for contradiction, that $\mathbb P\bigl(L_1(G_{n,m}^{\alpha,\ast})>(1+\delta)c_n\bigr)\not\to0.$ Then there exist a subsequence, which we do not relabel, and a constant $\eta>0$ such that
\[
\mathbb P\bigl(L_1(G_{n,m}^{\alpha,\ast})>(1+\delta)c_n\bigr)\ge 2\eta
\]
for all $n$ in the subsequence. Since $\mathbb P\bigl(d(G_{n,m}^{\alpha,\ast})\in\mathcal D_n\bigr)=1-o(1),$ it follows that, for all sufficiently large $n$ in the subsequence,
\[
\mathbb P\bigl(L_1(G_{n,m}^{\alpha,\ast})>(1+\delta)c_n,\ d(G_{n,m}^{\alpha,\ast})\in\mathcal D_n\bigr)\ge \eta.
\]
Hence for each such $n$ there exists a deterministic degree sequence $d_n\in\mathcal D_n$ such that
\[
\mathbb P\bigl(L_1(G_{n,m}^{\alpha,\ast})>(1+\delta)c_n \,\big|\, d(G_{n,m}^{\alpha,\ast})=d_n\bigr)\ge \eta.
\]
By Lemma~\ref{lem:jw-cm}, given $d(G_{n,m}^{\alpha,\ast})=d_n$, the multigraph $G_{n,m}^{\alpha,\ast}$
has the same law as $\mathrm{CM}(d_n)$. Therefore
\[
\mathbb P\bigl(L_1(\mathrm{CM}(d_n))>(1+\delta)c_n\bigr)\ge \eta.
\]
On the other hand, every $d_n\in\mathcal D_n$ satisfies the hypotheses of Proposition~\ref{prop:cp-upper} by Lemma~\ref{lem:good-degree-unified},
and Lemma~\ref{lem:saddle-unified} gives $T_n=(1+o(1))c_n$, and $\frac{\theta_n\ell(d_n)}{T_n}\to\infty.$ Applying Proposition~\ref{prop:cp-upper} with tolerance $\delta/2$, we obtain
\[
\mathbb P\bigl(L_1(\mathrm{CM}(d_n))>(1+\delta/2)T_n\bigr)=o(1).
\]
Since $T_n/c_n\to1$, we have
\[
T_n\le \frac{1+\delta}{1+\delta/2}\,c_n
\]
for all large $n$, and therefore
\[
(1+\delta/2)T_n\le (1+\delta)c_n.
\]
It follows that
\[
\mathbb P\bigl(L_1(\mathrm{CM}(d_n))>(1+\delta)c_n\bigr)=o(1),
\]
which contradicts the lower bound $\mathbb P(L_1(\mathrm{CM}(d_n))>(1+\delta)c_n)\ge \eta$. Therefore,
\[
\mathbb P\bigl(L_1(G_{n,m}^{\alpha,\ast})>(1+\delta)c_n\bigr)\to0.
\]

It remains to transfer the estimate from the multigraph process to the simplegraph process. Let $\mathcal G_n^+$ be the family of simple graphs on $[n]$
with exactly $m$ edges and largest component larger than $b_n$, that is,
\[
\mathcal G_n^+:=\{G:\ G\text{ is simple on }[n],\ e(G)=m,\ L_1(G)>b_n\}.
\]
Then
\[
\{L_1(G_{n,m}^{\alpha})>(1+\delta)c_n\}=\{G_{n,m}^{\alpha}\in \mathcal G_n^+\}.
\]
Since $m=\frac{\alpha}{2(\alpha+1)}(1-\varepsilon)n\le \frac n2$ and $\alpha\ge \alpha_*$ for all large $n$, Lemma~\ref{lem:jw-transfer} applies with $C=1$ and $\alpha_0=\alpha_*$. Hence there exists $B_0=B_0(1,\alpha_*)<\infty$ such that
\[
\mathbb P(G_{n,m}^{\alpha}\in \mathcal G_n^+)
\le
B_0\,\mathbb P(G_{n,m}^{\alpha,\ast}\in \mathcal G_n^+)+o(1).
\]
Since every graph in \(\mathcal G_n^+\) is simple,
\[
\mathbb P(G_{n,m}^{\alpha,\ast}\in \mathcal G_n^+)
\le
\mathbb P\bigl(L_1(G_{n,m}^{\alpha,\ast})>(1+\delta)c_n\bigr)\to0.
\]
Therefore
\[
\mathbb P\bigl(L_1(G_{n,m}^{\alpha})>(1+\delta)c_n\bigr)
=
\mathbb P(G_{n,m}^{\alpha}\in \mathcal G_n^+)
\to0.
\]
\end{proof}

\subsection{First moment estimate and regeneration}

Recall $Y^{(n),\star}$ denote the size biased version of $Y^{(n)}$, that is,
\[
\mathbb P\bigl(Y^{(n),\star}=r\bigr)=\frac{r\,\mathbb P(Y^{(n)}=r)}{\mu_n},\qquad r\ge 1.
\]
Let $(Y_i)_{i\ge1}$ be i.i.d. copies of $Y^{(n)}$, let $(Y_i^\star)_{i\ge1}$ be i.i.d. copies of $Y^{(n),\star}$, 
and set $S_r=\sum_{i=1}^r Y_i.$ For $2\le k\le m+1$, define
\[
a_{n,k}=\frac{\mu_n}{k-1}\,\mathbb P\!\left(\sum_{i=1}^k Y_i^\star=2k-2\right)
\]
and
\[
\Xi_{n,k}=
\left[
\prod_{j=1}^{k-1}\frac{(n-j)\mu_n}{\ell-(2j-1)}
\right]
\frac{\mathbb P(S_{n-k}=\ell-2k+2)}{\mathbb P(S_n=\ell)}.
\]

\begin{proposition}\label{prop:tree-first-moment} Recall that $C_k^{(n,m)}$ denotes the number of tree components of $G_{n,m}^{\alpha,*}$ having $k$ vertices.
Then
\[
\mathbb E C_k^{(n,m)}=\frac{n}{k}\,a_{n,k}\,\Xi_{n,k}.
\]
Moreover,
\[
a_{n,k}=
\frac{\alpha}{(\alpha+2)k-2}
\binom{(\alpha+2)k-2}{k-1}
 p_n^{k-1}(1-p_n)^{(\alpha+1)k-1}.
\]
\end{proposition}

\begin{proof}
Fix $d=(d_v)_{v\in[n]}$ with total degree $\ell$, and let $\mathbb P_d$ and
$\mathbb E_d$ denote probability and expectation under $\mathrm{CM}(d)$.

Fix a set $A\subseteq[n]$ with $|A|=k$. We first compute the probability that $A$ is a tree component of
$\mathrm{CM}(d)$. If $\sum_{v\in A} d_v\ne 2k-2,$ then $A$ cannot be a tree component, so the probability is $0$. 

Assume now that $\sum_{v\in A} d_v=2k-2.$ By Pr{\"u}fer formula, the number of labelled trees on the vertex set $A$ with vertex degrees $(d_v)_{v\in A}$ equals $\frac{(k-2)!}{\prod_{v\in A}(d_v-1)!}.$ Fix such a tree $T$. At each vertex $v\in A$, the $d_v$ edges of $T$ incident to $v$ can be assigned to the $d_v$
half-edges at $v$ in exactly $d_v!$ ways. Thus, over all vertices, there are $\prod_{v\in A} d_v!$ possible
half-edge assignments. 

Summing over all trees on $A$ with degree sequence $(d_v)_{v\in A}$, the total number of
internal pairings on the half-edges of $A$ that produce a tree on $A$ is
\[
\frac{(k-2)!}{\prod_{v\in A}(d_v-1)!}\prod_{v\in A} d_v!
=(k-2)!\prod_{v\in A} d_v.
\]
Once the half-edges incident to $A$ have been paired internally, the remaining $\ell-2k+2$ outside half-edges may be paired arbitrarily. Hence the number of favourable global pairings is
\[
(k-2)!\prod_{v\in A} d_v\cdot (\ell-2k+1)!!.
\]
The total number of pairings of all $\ell$ half-edges is $(\ell-1)!!$, so
\[
\mathbb P_d(A\text{ is a tree component})
=
\frac{\#\{\text{favourable pairings}\}}{\#\{\text{all pairings}\}}=
\mathbf 1_{\{\sum_{v\in A} d_v=2k-2\}}
(k-2)!\prod_{v\in A} d_v\,
\frac{(\ell-2k+1)!!}{(\ell-1)!!}.
\]
Summing over all $k$-subsets $A\subseteq[n]$ gives
\[
\mathbb E_d C_k^{(n,m)}
=
\sum_{\substack{A\subseteq[n]\\ |A|=k}}
\mathbf 1_{\{\sum_{v\in A} d_v=2k-2\}}
(k-2)!\prod_{v\in A} d_v\,
\frac{(\ell-2k+1)!!}{(\ell-1)!!}.
\]
Now return to the random multigraph $G_{n,m}^{\alpha,\ast}$. By Lemma~\ref{lem:jw-degree-law}, its degree sequence has the same law as
\begin{equation}
\label{eq:degree-conditioning}
(D_1,\dots,D_n)\overset{d}= (Y_1,\dots,Y_n)\,\big|\,\{S_n=\ell\}.
\end{equation}

where $Y_1,\dots,Y_n$ are i.i.d. copies of $Y^{(n)}$. By exchangeability, every $k$-subset contributes equally in expectation, so
\[
\mathbb E C_k^{(n,m)}
=
\binom{n}{k}(k-2)!\frac{(\ell-2k+1)!!}{(\ell-1)!!}
\mathbb E\!\left[\prod_{i=1}^k D_i\,\mathbf 1_{\{\sum_{i=1}^k D_i=2k-2\}}\right].
\]
Using \eqref{eq:degree-conditioning}, we have
\[
\mathbb E\!\left[\prod_{i=1}^k D_i\,\mathbf 1_{\{\sum_{i=1}^k D_i=2k-2\}}\right]
=
\frac{\mathbb E\!\left[\prod_{i=1}^k Y_i\,\mathbf 1_{\{\sum_{i=1}^k Y_i=2k-2\}}\mathbf 1_{\{S_n=\ell\}}\right]}
{\mathbb P(S_n=\ell)}.
\]
On the event $\{\sum_{i=1}^k Y_i=2k-2\}$, the condition $S_n=\ell$ is equivalent to
$S_{n-k}=\ell-2k+2$. Since $Y_1,\dots,Y_k$ are independent of $Y_{k+1},\dots,Y_n$, we obtain
\[
\mathbb E\!\left[\prod_{i=1}^k Y_i\,\mathbf 1_{\{\sum_{i=1}^k Y_i=2k-2\}}\mathbf 1_{\{S_n=\ell\}}\right]
=
\mathbb E\!\left[\prod_{i=1}^k Y_i\,\mathbf 1_{\{\sum_{i=1}^k Y_i=2k-2\}}\right]
\mathbb P(S_{n-k}=\ell-2k+2).
\]
Using the size bias identity $\mathbb E\bigl[Y_i f(Y_i)\bigr]=\mu_n\,\mathbb E\bigl[f(Y_i^\star)\bigr]$ independently for $i=1,\dots,k$, we find
\[
\mathbb E\!\left[\prod_{i=1}^k Y_i\,\mathbf 1_{\{\sum_{i=1}^k Y_i=2k-2\}}\right]
=
\mu_n^k\,\mathbb P\!\left(\sum_{i=1}^k Y_i^\star=2k-2\right).
\]
Therefore
\[
\mathbb E C_k^{(n,m)}
=
\binom{n}{k}(k-2)!\frac{(\ell-2k+1)!!}{(\ell-1)!!}
\mu_n^k\,\mathbb P\!\left(\sum_{i=1}^k Y_i^\star=2k-2\right)
\frac{\mathbb P(S_{n-k}=\ell-2k+2)}{\mathbb P(S_n=\ell)}.
\]
Using
\begin{equation}
\label{eq:binom-factorial-product}
\binom{n}{k}(k-2)!
=
\frac{n}{k(k-1)}
\prod_{j=1}^{k-1}(n-j),
\end{equation}
and
\begin{equation}
\label{eq:double-factorial-ratio-product}
\frac{(\ell-2k+1)!!}{(\ell-1)!!}
=
\prod_{j=1}^{k-1}
\frac{1}{\ell-(2j-1)}.
\end{equation}
Substituting \eqref{eq:binom-factorial-product}, \eqref{eq:double-factorial-ratio-product}  and using the definitions of $a_{n,k}$ and $\Xi_{n,k}$ yields
\[
\mathbb E C_k^{(n,m)}=\frac{n}{k}\,a_{n,k}\,\Xi_{n,k}.
\]

It remains to compute $a_{n,k}$. Since $Y^{(n)}\sim\operatorname{NBin}(\alpha,p_n)$, the size biased law satisfies $Y^{(n),\star}-1\sim\operatorname{NBin}(\alpha+1,p_n).$ Hence
\[
\sum_{i=1}^k Y_i^\star-k\sim\operatorname{NBin}(k(\alpha+1),p_n).
\]
The event $\sum_{i=1}^k Y_i^\star=2k-2$ is equivalent to $\sum_{i=1}^k (Y_i^\star-1)=k-2$, so
\[
\mathbb P\!\left(\sum_{i=1}^k Y_i^\star=2k-2\right)
=
\binom{(\alpha+2)k-3}{k-2}(1-p_n)^{k(\alpha+1)}p_n^{k-2}.
\]
Since $\mu_n=\frac{\alpha p_n}{1-p_n},$ we obtain
\[
a_{n,k}
=
\frac{\alpha p_n}{(1-p_n)(k-1)}
\binom{(\alpha+2)k-3}{k-2}(1-p_n)^{k(\alpha+1)}p_n^{k-2}.
\]
Now consider $\binom{N-1}{K-1}=\frac{K}{N}\binom{N}{K}$ with $N=(\alpha+2)k-2$ and $K=k-1$, to obtain
\[
a_{n,k}
=
\frac{\alpha}{(\alpha+2)k-2}
\binom{(\alpha+2)k-2}{k-1}
p_n^{k-1}(1-p_n)^{(\alpha+1)k-1}.
\]
\end{proof}

The following result provides the conditional law of the graph
remaining after the removal of a prescribed tree component.


\begin{proposition}\label{prop:regeneration}
Fix $k\ge2$ and $A\subseteq[n]$ with $|A|=k$. Condition on the event that $A$ is a tree component of
$G_{n,m}^{\alpha,\ast}$. Then, after relabelling the remaining vertices $[n]\setminus A$ as $[n-k]$, the induced
multigraph on $[n]\setminus A$ has the distribution of $G_{n-k,m-k+1}^{\alpha,\ast}$.
\end{proposition}

\begin{proof}
We first condition on the full degree sequence $d=(d_v)_{v\in[n]}$. By Lemma~\ref{lem:jw-cm}, 
the multigraph $G_{n,m}^{\alpha,\ast}$ has distribution $\mathrm{CM}(d)$. Recall that, for a fixed set $A$, on the event that $A$ is a tree component, all half-edges incident to vertices of $A$ are paired inside $A$, and the remaining half-edges are paired among the vertices of $[n]\setminus A$.
For a fixed outside pairing, the number of internal pairings on the half-edges of $A$ that produce a tree equals $(k-2)!\prod_{v\in A} d_v$ when $\sum_{v\in A} d_v=2k-2$, and equals $0$ otherwise. 

In particular, conditional on the degree sequence $d$ and on the event that $A$ is a tree component, every outside pairing occurs with the same probability. Therefore the induced multigraph on $[n]\setminus A$ is the configuration model with outside degrees $(d_v)_{v\notin A}$.

It remains to identify the conditional law of the outside degree vector. Recall that $p_n=\frac{2m}{\alpha n+2m},$ and let $Y_1,\dots,Y_n$ be i.i.d. $\operatorname{NBin}(\alpha,p_n)$ random variables. By Lemma~\ref{lem:jw-degree-law},
\[
d\bigl(G_{n,m}^{\alpha,\ast}\bigr)\overset{d}= (Y_1,\dots,Y_n)\,\big|\,\{Y_1+\cdots+Y_n=2m\}.
\]

For vector $e=(e_v)_{v\in[n]}$ with total sum $2m$. The computation in the proof of Proposition~\ref{prop:tree-first-moment} shows that
\begin{equation*}
\begin{aligned}
\mathbb{P}_e(A \text{ is a tree component})
&=
\mathbf{1}_{\{\sum_{v \in A} e_v = 2k-2\}}
(k-2)!
\prod_{v \in A} e_v
\frac{(2m-2k+1)!!}{(2m-1)!!} \\
&=
\kappa_{n,k}
\mathbf{1}_{\{\sum_{v \in A} e_v = 2k-2\}}
\prod_{v \in A} e_v .
\end{aligned}
\end{equation*}
where
\[
\kappa_{n,k}=(k-2)!\frac{(2m-2k+1)!!}{(2m-1)!!}.
\]
Now fix an outside vector $f=(f_v)_{v\notin A}$ with sum $2m-2k+2$. Summing over all inside vectors
$(e_v)_{v\in A}$ with $2k-2$, we obtain
\begin{equation}
\label{eq:conditional-degree-given-tree-component}
\begin{aligned}
&\mathbb{P}\bigl((D_v)_{v\notin A}=f \mid A\text{ is a tree component}\bigr) \\
&\qquad\propto
\left[\prod_{v\notin A}\mathbb{P}(Y_v=f_v)\right]
\sum_{\substack{(e_v)_{v\in A}\in\mathbb{N}_0^A\\
\sum_{v\in A} e_v=2k-2}}
\kappa_{n,k}
\Bigl(\prod_{v\in A} e_v\Bigr)
\left[\prod_{v\in A}\mathbb{P}(Y_v=e_v)\right].
\end{aligned}
\end{equation}
The sum over the inside coordinates in \eqref{eq:conditional-degree-given-tree-component} does not depend on $f$. Consequently, conditional on the total degree $2m$ and on the event that $A$ is a tree component, the outside degree vector is distributed as i.i.d.
$\operatorname{NBin}(\alpha,p_n)$ random variables conditioned on having total sum $2m-2k+2$.

Now let $Z_1,\dots,Z_{n-k}$ be i.i.d. $\operatorname{NBin}(\alpha,p)$ random variables with arbitrary $p\in(0,1)$.
For any vector $z=(z_1,\dots,z_{n-k})\in\mathbb N_0^{n-k}$, we have
\[
\mathbb P\!\left((Z_1,\dots,Z_{n-k})=z\,\middle|\,\sum_{i=1}^{n-k} Z_i=s\right)
=
\frac{\mathbf 1_{\{\sum_i z_i=s\}}\prod_{i=1}^{n-k}\binom{\alpha+z_i-1}{z_i}}
{\displaystyle\sum_{\substack{u\in\mathbb N_0^{n-k}\\ \sum_i u_i=s}}\prod_{i=1}^{n-k}\binom{\alpha+u_i-1}{u_i}}.
\]
Indeed, on the conditioning event the common factor $(1-p)^{\alpha(n-k)}p^s$ cancels. Thus the conditional law is independent of $p$.

Taking $p'=\frac{2(m-k+1)}{\alpha(n-k)+2(m-k+1)},$ Lemma~\ref{lem:jw-degree-law} implies that the outside degree vector has  the degree-sequence law of
$G_{n-k,m-k+1}^{\alpha,\ast}$. Together with the conditional configuration model description established above, we obtain
\[
G_{n,m}^{\alpha,*}[A^c] \mid \{A \text{ is a tree component}\}
\overset{d}= G_{n-k,m-k+1}^{\alpha,*}.
\]
\end{proof}

\subsection{Lower bound for Theorem~\ref{thm:main-fixed-alpha}}

\begin{lemma}\label{lem:xi-unified}
Fix $B<\infty$. Uniformly for integers $2\le k\le B\varepsilon^{-2}\log x_n,$ we have
\[
\Xi_{n,k}=1+o(1).
\]
\end{lemma}

\begin{proof}
Write $\Xi_{n,k}=U_{n,k}V_{n,k},$ where
\[
U_{n,k}=\prod_{j=1}^{k-1}\frac{(n-j)\mu_n}{\ell-(2j-1)}
=
\prod_{j=1}^{k-1}\frac{1-j/n}{1-(2j-1)/\ell},
\]
and
\[
V_{n,k}=
\frac{\mathbb P(S_{n-k}=\ell-2k+2)}{\mathbb P(S_n=\ell)}.
\]

Since $k\le B\varepsilon^{-2}\log x_n$ and $x_n=\varepsilon^3 n\to\infty$,
\[
\frac{k}{n}
=
O\!\left(\frac{\varepsilon\log x_n}{x_n}\right)
=o(1),
\qquad
\frac{k}{n^{2/3}}
=
O\!\left(\frac{\log x_n}{x_n^{2/3}}\right)
=o(1).
\]
In particular, $k=o(n^{2/3})$ and $k=o(n)$.

We begin with $U_{n,k}$. Since $j/n=o(1)$ and $(2j-1)/\ell=o(1)$ uniformly for $1\le j\le k-1$,
the expansion $\log(1-u)=-u+O(u^2)$ gives
\[
\begin{aligned}
\log U_{n,k}
&=
\sum_{j=1}^{k-1}
\left(
-\frac{j}{n}
+
\frac{2j-1}{\ell}
\right)
+
O\left(
\sum_{j=1}^{k-1}\frac{j^2}{n^2}
+
\sum_{j=1}^{k-1}\frac{j^2}{\ell^2}
\right) \\
&=
\sum_{j=1}^{k-1}
\left(
-\frac{j}{n}
+
\frac{2j-1}{\ell}
\right)
+
O\left(
\sum_{j=1}^{k-1}\frac{j^2}{n^2}
\right).
\end{aligned}
\]
Note that
\[
\sum_{j=1}^{k-1} j=\frac{k^2}{2}+O(k),
\qquad
\sum_{j=1}^{k-1} (2j-1)=k^2+O(k),
\qquad
\sum_{j=1}^{k-1} j^2=O(k^3),
\]
and $\ell=\mu_n n$, so
\begin{equation}\label{eq:logU}
\log U_{n,k}=
-\frac{1}{n}
\left(
\frac{k^2}{2}
+
O(k)
\right)
+
\frac{1}{\ell}
\left(
k^2
+
O(k)
\right)
+
O\left(
\frac{k^3}{n^2}
\right)
=
\left(\frac1{\mu_n}-\frac12\right)\frac{k^2}{n}
+
O\!\left(\frac{k}{n}+\frac{k^3}{n^2}\right).
\end{equation}

We next estimate $V_{n,k}$ in a way that is uniform in $\alpha$.  Recall that $\ell=\mu_n n,$ set $q_k=(\alpha+2)k-2.$ Since $S_r\sim \operatorname{NBin}(r\alpha,p_n)$, we have
\[
V_{n,k}
=
\frac{\Gamma(n\alpha)}{\Gamma(n\alpha-\alpha k)}
\cdot
\frac{\Gamma(n\alpha+\ell-v)}{\Gamma(n\alpha+\ell)}
\cdot
\frac{\Gamma(\ell+1)}{\Gamma(\ell-(2k-2)+1)}
\cdot
(1-p_n)^{-\alpha k}p_n^{-(2k-2)}.
\]
Define
\[
h(z)=\left(z-\frac12\right)\log z-z,
\]
and $$\varphi(y)=y+(1-y)\log(1-y)\qquad (0\le y<1).$$
By Stirling's formula,
\[
\log\Gamma(z)=h(z)+\frac12\log(2\pi)+r(z),
\]
where $r(z)=O(z^{-1}),$
uniformly for $z\ge1$.
We claim that, uniformly whenever $0\le u\le z/2$,
\begin{equation}\label{eq:h-diff}
h(z-u)-h(z)=-u\log z+z\varphi(u/z)+O(u/z).
\end{equation}
Indeed,
\[
h(z-u)-h(z)
=
-u\log z+\left(z-u-\frac12\right)\log\!\left(1-\frac{u}{z}\right)+u.
\]
The last two terms equal $z\varphi(u/z)-\frac12\log(1-u/z),$ and since $u/z\le1/2$, we have $|\log(1-u/z)|=O(u/z)$, this proves \eqref{eq:h-diff}.

Now $\frac{\alpha k}{n\alpha}=\frac{k}{n}=o(1),$ $\frac{q_k}{n\alpha+\ell}=O\!\left(\frac{k}{n}\right)=o(1),$ and $\frac{2k-2}{\ell}=O\!\left(\frac{k}{n}\right)=o(1),$ and therefore $u\le z/2$ in each of the three gamma ratio applications below. Using \eqref{eq:h-diff}
and $r(z)=O(z^{-1})$, we obtain
\[
\log\frac{\Gamma(n\alpha)}{\Gamma(n\alpha-\alpha k)}
=
\alpha k\log n\alpha-n\alpha\varphi(\alpha k/n\alpha)+O(k/n),
\]
\[
\log\frac{\Gamma(n\alpha+\ell-v)}{\Gamma(n\alpha+\ell)}
=
-q_k\log(n\alpha+\ell)+(n\alpha+\ell)\varphi(q_k/(n\alpha+\ell))+O(k/n),
\]
and
\[
\log\frac{\Gamma(\ell+1)}{\Gamma(\ell-2k+3)}
=
(2k-2)\log \ell-\ell\varphi((2k-2)/\ell)+O(k/n).
\]
Therefore
\begin{align}
\log V_{n,k}
={}&
\alpha k\log n\alpha-q_k\log(n\alpha+\ell)+(2k-2)\log \ell-\alpha k\log(1-p_n)-(2k-2)\log p_n
\notag\\
&\hspace{4em}
-n\alpha\varphi(\alpha k/n\alpha)+(n\alpha+\ell)\varphi(q_k/(n\alpha+\ell))-\ell\varphi((2k-2)/\ell)+O(k/n).
\label{eq:logV-raw}
\end{align}
Since $1-p_n=\frac{n\alpha}{n\alpha+\ell},$ $p_n=\frac{\ell}{n\alpha+\ell},$ $q_k=\alpha k+2k-2,$ the logarithmic terms in \eqref{eq:logV-raw} cancel exactly. Thus
\begin{equation}\label{eq:logV-phi}
\log V_{n,k}
=
-n\alpha\varphi(\alpha k/n\alpha)+(n\alpha+\ell)\varphi(q_k/(n\alpha+\ell))-\ell\varphi((2k-2)/\ell)+O(k/n).
\end{equation}

Replacing $v$ by $(\alpha+2)k$ and $(2k-2)$ by $2k$ at a total cost $O(k/n)$.
Indeed, $\varphi'(y)=-\log(1-y)=O(y)$ uniformly for $0\le y\le Ck/n$, and each replacement changes
the relevant argument of $\varphi$ by $O(1/n)$. Hence, by the mean value theorem,
\[
(n\alpha+\ell)\left|
\varphi\!\left(\frac{q_k}{n\alpha+\ell}\right)-\varphi\!\left(\frac{(\alpha+2)k}{n\alpha+\ell}\right)
\right|
=O(k/n),
\]
and similarly
\[
\ell\left|
\varphi\!\left(\frac{2k-2}{\ell}\right)-\varphi\!\left(\frac{2k}{\ell}\right)
\right|
=O(k/n).
\]
Therefore
\begin{equation}\label{eq:logV-F}
\log V_{n,k}=nF_n(k/n)+O(k/n),
\end{equation}
where
\[
F_n(x)=
-\alpha\varphi(x)+(\alpha+\mu_n)\varphi\!\left(\frac{(\alpha+2)x}{\alpha+\mu_n}\right)
-\mu_n\varphi\!\left(\frac{2x}{\mu_n}\right).
\]

Note that $F_n(0)=0,$ $F_n'(0)=0.$ A direct calculation gives
\[
F_n''(0)
=
-\alpha+\frac{(\alpha+2)^2}{\alpha+\mu_n}-\frac{4}{\mu_n}.
\]
Hence
\begin{align}
\frac12 F_n''(0)
&=
-\frac{\alpha}{2}+\frac{(\alpha+2)^2}{2(\alpha+\mu_n)}-\frac{2}{\mu_n}
\notag\\
&=
-\left(\frac1{\mu_n}-\frac12\right)
-\frac{(\mu_n-2)\bigl((\alpha+1)\mu_n-\alpha\bigr)}{2\mu_n(\alpha+\mu_n)}.
\label{eq:Fpp}
\end{align}
Now
\[
(\alpha+1)\mu_n-\alpha
=
(\alpha+1)\frac{\alpha}{\alpha+1}(1-\varepsilon)-\alpha
=
-\alpha\varepsilon,
\]
so the second term on the right hand side of \eqref{eq:Fpp} is $O(\varepsilon)$ uniformly in $n$.
Therefore
\begin{equation}\label{eq:Fpp-simplified}
\frac12 F_n''(0)
=
-\left(\frac1{\mu_n}-\frac12\right)+O(\varepsilon).
\end{equation}

We next bound $F_n'''$ uniformly on the range of interest. Differentiating yields
\[
F_n'''(x)
=
-\frac{\alpha}{(1-x)^2}
+
\frac{(\alpha+2)^3}{\bigl(\alpha+\mu_n-(\alpha+2)x\bigr)^2}
-
\frac{8}{(\mu_n-2x)^2}.
\]
For $0\le x\le \frac{B\varepsilon^{-2}\log x_n}{n},$ we have $x=o(1)$, and since $\mu_n\to a/(a+1)$ when $a<\infty$ and $\mu_n\to1$ when $a=\infty$,
the quantity $\mu_n$ is bounded away from $0$. Hence the last term is $O(1)$ uniformly.

For the first two terms, set $b_1=\frac{\mu_n-2}{\alpha+2}.$ Then
\[
\alpha+\mu_n-(\alpha+2)x=(\alpha+2)(1-x+b_1),
\]
and therefore
\[
-\frac{\alpha}{(1-x)^2}
+
\frac{\alpha+2}{(1-x+b_1)^2}
=
\frac{2(1-x)^2-2\alpha (1-x) b_1-\alpha b_1^2}{(1-x)^2(1-x+b_1)^2}.
\]
Now $1-x$ is bounded away from $0$, while
\[
\alpha b_1=\frac{\alpha(\mu_n-2)}{\alpha+2}=O(1),
\]
and $$\alpha b_1^2=\frac{\alpha(\mu_n-2)^2}{(\alpha+2)^2}=O(1)$$
uniformly in $n$. Thus $F_n'''(x)=O(1)$ uniformly over \(x\).

Taylor's expansion therefore gives
\[
F_n(x)=\frac12F_n''(0)x^2+O(x^3)
\]
uniformly for $0\le x\le B\varepsilon^{-2}\log x_n/n$. Taking $x=k/n$ and using
\eqref{eq:logV-F} and \eqref{eq:Fpp-simplified}, we obtain
\begin{equation}\label{eq:logV}
\log V_{n,k}
=
-\left(\frac1{\mu_n}-\frac12\right)\frac{k^2}{n}
+
O\!\left(
\varepsilon\frac{k^2}{n}+\frac{k}{n}+\frac{k^3}{n^2}
\right).
\end{equation}

Combining \eqref{eq:logU} and \eqref{eq:logV}, we get
\[
\log \Xi_{n,k}
=
O\!\left(
\varepsilon\frac{k^2}{n}+\frac{k}{n}+\frac{k^3}{n^2}
\right).
\]
Finally,
\[
\varepsilon\frac{k^2}{n}
=
O\!\left(\frac{\log^2 x_n}{x_n}\right)=o(1),
\qquad
\frac{k}{n}
=
O\!\left(\frac{\varepsilon\log x_n}{x_n}\right)=o(1),
\qquad
\frac{k^3}{n^2}
=
O\!\left(\frac{\log^3 x_n}{x_n^2}\right)=o(1),
\]
uniformly in the present range. Hence $\log \Xi_{n,k}=o(1)$ uniformly, and therefore
\[
\Xi_{n,k}=1+o(1)
\]
uniformly for $2\le k\le B\varepsilon^{-2}\log x_n$.
\end{proof}

The following lemma gives the leading order estimate for the expected
number of tree components of size \(k\).

\begin{lemma}\label{lem:ank-unified}
Let
\[
\Psi_n=\log\frac{1}{1-\varepsilon}+(\alpha+2)\log\!\left(1-\frac{\varepsilon}{\alpha+2}\right).
\]
Then
\[
\Psi_n=\frac{\alpha+1}{2(\alpha+2)}\varepsilon^2+O(\varepsilon^3),
\]
where the \(O(\varepsilon^3)\) remainder is uniform in \(n\). Moreover, for every fixed \(B<\infty\) there exist constants
\(0<c_1<c_2<\infty\), depending only on \(B\) and \(\alpha_*\), such that, for all sufficiently large \(n\),
uniformly for $2\le k\le B\varepsilon^{-2}\log x_n,$ we have
\[
c_1 k^{-3/2}e^{-\Psi_n k}\le a_{n,k}\le c_2 k^{-3/2}e^{-\Psi_n k}.
\]
Consequently,
\[
\mathbb E C_k^{(n,m)}=\Theta\!\bigl(nk^{-5/2}e^{-\Psi_n k}\bigr).
\]
\end{lemma}

\begin{proof}
Since \(\varepsilon\to0\), we may assume throughout that \(\varepsilon\le 1/2\). By the Taylor expansions
\[
\log\frac1{1-\varepsilon}
=
\varepsilon+\frac{\varepsilon^2}{2}+O(\varepsilon^3)
\]
and
\[
(\alpha+2)\log\!\left(1-\frac{\varepsilon}{\alpha+2}\right)
=
-\varepsilon-\frac{\varepsilon^2}{2(\alpha+2)}+O\!\left(\frac{\varepsilon^3}{(\alpha+2)^2}\right),
\]
we obtain
\[
\Psi_n
=
\frac{\alpha+1}{2(\alpha+2)}\varepsilon^2+O(\varepsilon^3).
\]
The remainder is uniform because \(\alpha+2\ge \alpha_*+2>0\).

We next estimate \(a_{n,k}\). By Proposition~\ref{prop:tree-first-moment},
\[
a_{n,k}
=
\frac{\alpha}{q_k}\binom{q_k}{k-1}p_n^{k-1}(1-p_n)^L,
\]
where $q_k=(\alpha+2)k-2,$ and $L=q_k-k+1=(\alpha+1)k-1$. Note that \(k-1\ge1\), \(L\ge 2\alpha_*+1>1\), and \(q_k\ge 2\alpha_*+2>1\) for every \(k\ge2\).

For the generalized binomial coefficient for possibly noninteger \(q_k\) and \(L\), define
\[
g(x)=\Gamma(x+1)e^x x^{-x-1/2},\qquad x\ge1.
\]
The function \(g\) is continuous and strictly positive on \([1,\infty)\), and Stirling's formula gives
\(g(x)\to\sqrt{2\pi}\) as \(x\to\infty\). Hence there exist absolute constants \(0<c_\Gamma<C_\Gamma<\infty\)
such that
\begin{equation}
c_\Gamma x^{x+1/2}e^{-x} \le \Gamma(x+1) \le C_\Gamma x^{x+1/2}e^{-x}, \qquad x \ge 1.
\label{eq:gamma_bound}
\end{equation}
Applying \eqref{eq:gamma_bound} with \(x=q_k,k-1,L\), and recalling that
\[
\binom{q_k}{k-1}=\frac{\Gamma(q_k+1)}{\Gamma(k)\Gamma(L+1)},
\]
we obtain
\[
\binom{q_k}{k-1}
=
\Theta\!\left(\sqrt{\frac{q_k}{(k-1)L}}\,\frac{q_k^{q_k}}{(k-1)^{k-1} L^L}\right),
\]
uniformly for all \(k\ge2\). Therefore
\begin{equation}
a_{n,k} = \Theta\!\left(\frac{\alpha}{\sqrt{q_k(k-1)L}}\,A_{n,k}\right),
\label{eq:a_nk}
\end{equation}
where $$
A_{n,k}=\left(\frac{q_kp_n}{k-1}\right)^{k-1}\left(\frac{q_k(1-p_n)}{L}\right)^L.
$$

We first estimate the prefactor in \eqref{eq:a_nk}. Since \(k\ge2\), $\frac{k}{2}\le k-1\le k, \frac{(\alpha+1)k}{2}\le L\le (\alpha+1)k,$ and $\frac{(\alpha+2)k}{2}\le q_k\le (\alpha+2)k.$  Hence
\[
\frac{\alpha}{\sqrt{(\alpha+2)(\alpha+1)}}\;k^{-3/2}
\le
\frac{\alpha}{\sqrt{q_k(k-1)L}}
\le
\frac{2\sqrt2\,\alpha}{\sqrt{(\alpha+2)(\alpha+1)}}\;k^{-3/2}.
\]
Since \(\alpha\ge \alpha_*\),
\[
\frac{\alpha}{\sqrt{(\alpha+2)(\alpha+1)}}
\ge
\frac{\alpha}{\alpha+2}
\ge
\frac{\alpha_*}{\alpha_*+2},
\]
and \(\alpha/\sqrt{(\alpha+2)(\alpha+1)}\le1\). Therefore
\[
\frac{\alpha}{\sqrt{q_k(k-1)L}}=\Theta(k^{-3/2}).
\]

It remains to estimate \(A_{n,k}\). Taking logarithms gives
\[
\log A_{n,k}
=
(k-1)\log\frac{q_k}{k-1}+L\log\frac{q_k}{L}+(k-1)\log p_n+L\log(1-p_n).
\]

We estimate these terms as follows. First,
\[
\frac{q_k}{k-1}
=
(\alpha+2)\left(1+\frac{\alpha}{(\alpha+2)(k-1)}\right).
\]
Note that \(0\le \frac{\alpha}{(\alpha+2)(k-1)}\le 1/(k-1)\le1\), so
\[
0\le (k-1)\log(1+\frac{\alpha}{(\alpha+2)(k-1)})\le (k-1)\frac{\alpha}{(\alpha+2)(k-1)}=\frac{\alpha}{\alpha+2}\le1.
\]
Consequently, uniformly in \(k\),
\begin{equation}\label{eq:KlogNK}
(k-1)\log\frac{q_k}{k-1}
=
(k-1)\log(\alpha+2)+O(1).
\end{equation}


Second,
\[
\frac{q_k}{L}
=
\frac{\alpha+2}{\alpha+1}\,r_{n,k},
\]
where $$r_{n,k}=\frac{1-\frac{2}{(\alpha+2)k}}{1-\frac{1}{(\alpha+1)k}}.$$
A simple calculation gives
\[
r_{n,k}=1-\frac{\alpha}{(\alpha+2)\bigl((\alpha+1)k-1\bigr)}.
\]
Hence
\[
|r_{n,k}-1|
\le
\frac{1}{(\alpha+1)k}
\le \frac12.
\]
Since \(|\log(1-z)|\le 2|z|\) for \(|z|\le1/2\), it follows that
\[
|\log r_{n,k}|
\le
\frac{2\alpha}{(\alpha+2)\bigl((\alpha+1)k-1\bigr)}.
\]
Multiplying by \(L=(\alpha+1)k-1\), we obtain $L|\log r_{n,k}|\le \frac{2\alpha}{(\alpha+2)}\le2.$ Therefore, uniformly in \(k\),
\begin{equation}\label{eq:LlogNL}
L\log\frac{q_k}{L}
=
((\alpha+1)k-1)\log\frac{\alpha+2}{\alpha+1}+O(1).
\end{equation}

Third, since $1-p_n=\frac{\alpha+1}{\alpha+2-\varepsilon},$
we have
\begin{equation}\label{eq:p-log-terms}
\begin{aligned}
(k-1)\log p_n+L\log(1-p_n)
&=
(k-1)\log(1-\varepsilon)
+
((\alpha+1)k-1)\log(\alpha+1) \\
&\quad
-
((\alpha+2)k-2)\log(\alpha+2-\varepsilon).
\end{aligned}
\end{equation}
Combining \eqref{eq:KlogNK}, \eqref{eq:LlogNL}, and \eqref{eq:p-log-terms}, the terms involving \(\log(\alpha+1)\) cancel, and we get
\[
\log A_{n,k}
=
((\alpha+2)k-2)\log\frac{\alpha+2}{\alpha+2-\varepsilon}+(k-1)\log(1-\varepsilon)+O(1).
\]
Now
\[
\log\frac{\alpha+2}{\alpha+2-\varepsilon}=-\log\!\left(1-\frac{\varepsilon}{\alpha+2}\right)=O\!\left(\frac{\varepsilon}{\alpha+2}\right),
\qquad
\log(1-\varepsilon)=O(\varepsilon),
\]
so replacing \(((\alpha+2)k-2)\) by \((\alpha+2)k\) and \((k-1)\) by \(k\) only changes the expression by
\(O(1)\). Therefore
\[
\log A_{n,k}
=
-(\alpha+2)k\log\!\left(1-\frac{\varepsilon}{\alpha+2}\right)+k\log(1-\varepsilon)+O(1)
=
-\Psi_n k+O(1)
\]
uniformly for \(2\le k\le B\varepsilon^{-2}\log x_n\). Exponentiating gives \(A_{n,k}=\Theta(e^{-\Psi_n k})\) uniformly. Thus, we have $a_{n,k}=\Theta\!\bigl(k^{-3/2}e^{-\Psi_n k}\bigr)$ uniformly for \(2\le k\le B\varepsilon^{-2}\log x_n\).

Finally, Proposition~\ref{prop:tree-first-moment} and Lemma~\ref{lem:xi-unified} yield
\[
\mathbb E C_k^{(n,m)}
=
\frac{n}{k}\,a_{n,k}\,\Xi_{n,k}.
\]
Since \(\Xi_{n,k}=1+o(1)\) uniformly in the same range, it is bounded between \(1/2\) and \(2\) for all
sufficiently large \(n\), and therefore
\[
\mathbb E C_k^{(n,m)}=\Theta\!\bigl(nk^{-5/2}e^{-\Psi_n k}\bigr)
\]
uniformly for \(2\le k\le B\varepsilon^{-2}\log x_n\). 
\end{proof}

After deleting a tree component of moderate size, the
expectation remains almost unchanged. Together with Proposition~\ref{prop:regeneration},
Lemma~\ref{lem:compare-expect} then implies that the correlation between two tree
components of moderate size is weak.
\begin{lemma}\label{lem:compare-expect}
Fix \(B\in(0,\infty)\). Uniformly over all integers $2\le j,k\le B\varepsilon^{-2}\log x_n,$ we have
\[
\frac{\mathbb E C_j^{(n-k,m-k+1)}}{\mathbb E C_j^{(n,m)}}=1+o(1).
\]
\end{lemma}

\begin{proof}
Set $n'=n-k, m'=m-k+1, \ell'=2m'.$ Write $m'=\frac{\alpha n'}{2(\alpha+1)}(1-\varepsilon'),$ so that
\[
\varepsilon'=1-\frac{2(\alpha+1)m'}{\alpha n'}.
\]
Since $\mu_n'= \frac{2m'}{n'}=\frac{\ell-2k+2}{n-k},$ and $\mu_n=\frac{\ell}{n},$ we have
\[
\mu_n'-\mu_n
=
\frac{\ell-2k+2}{n-k}-\frac{\ell}{n}
=
\frac{k(\ell-2n)+2n}{n(n-k)}
=
O\!\left(\frac{k}{n}\right),
\]
where the last equation uses $\ell=O(n)$. Since $\varepsilon=1-\frac{\alpha+1}{\alpha}\mu_n,$ and $\varepsilon'=1-\frac{\alpha+1}{\alpha}\mu_n',$
it follows that
\begin{equation}\label{eq:eps-prime}
\varepsilon'-\varepsilon=O\!\left(\frac{k}{n}\right)
\end{equation}
uniformly in the stated range.

Moreover, $\frac{k}{\varepsilon n}
=
O\!\left(\frac{\log x_n}{x_n}\right)=o(1),$ so $\varepsilon'\sim \varepsilon$. In particular, $(\varepsilon')^3 n'
=
x_n(1+o(1)).$ Applying Proposition~\ref{prop:tree-first-moment} to $(n,m)$ and $(n',m')$, we obtain
\begin{equation}
\label{eq:expectation-ratio-deleted-component}
\frac{\mathbb{E} C_j^{(n-k,m-k+1)}}{\mathbb{E} C_j^{(n,m)}}
=
\frac{n'}{n}
\cdot
\frac{a'_{n,j}}{a_{n,j}}
\cdot
\frac{\Xi'_{n,j}}{\Xi_{n,j}}.
\end{equation}
where the primed quantities are defined with $(n',m')$ in place of $(n,m)$.
The first factor in \eqref{eq:expectation-ratio-deleted-component} satisfies
\[
\frac{n'}{n}=1+O\!\left(\frac{k}{n}\right)=1+o(1).
\]
For the last factor, Lemma~\ref{lem:xi-unified} applies to both $(n,m)$ and
$(n',m')$, since $(\varepsilon')^3 n'\to\infty$ and the same bound on $j$
remains valid.

Hence \(\Xi_{n,j}=1+o(1)\) and \(\Xi'_{n,j}=1+o(1)\) uniformly, and therefore

\[
\frac{\Xi'_{n,j}}{\Xi_{n,j}}=1+o(1).
\]

It remains to compare $a'_{n,j}$ and $a_{n,j}$. By Proposition~\ref{prop:tree-first-moment},
\[
a_{n,j}
=
\frac{\alpha}{(\alpha+2)j-2}
\binom{(\alpha+2)j-2}{j-1}
p_n^{j-1}(1-p_n)^{(\alpha+1)j-1}.
\]
Recall that $1-p_n=\frac{\alpha+1}{\alpha+2-\varepsilon},$
we may write
\[
\log a_{n,j}
=
c_j+(j-1)\log(1-\varepsilon)-((\alpha+2)j-2)\log(\alpha+2-\varepsilon),
\]
where $c_j$ depends on $j$ and $\alpha$, but not on $\varepsilon$.
Differentiating with respect to $\varepsilon$ gives
\[
\frac{\partial}{\partial\varepsilon}\log a_{n,j}
=
-\frac{j-1}{1-\varepsilon}+\frac{(\alpha+2)j-2}{\alpha+2-\varepsilon}
=
-\frac{(\alpha+1)j\varepsilon}{(1-\varepsilon)(\alpha+2-\varepsilon)}+O(1)
=
O(j\varepsilon+1)
\]
uniformly for small $\varepsilon$. By the mean value theorem and \eqref{eq:eps-prime},
\[
\log\frac{a'_{n,j}}{a_{n,j}}
=
O\bigl((j\varepsilon+1)|\varepsilon'-\varepsilon|\bigr)
=
O\!\left((j\varepsilon+1)\frac{k}{n}\right).
\]
Note that $j,k\le B\varepsilon^{-2}\log x_n$,
\[
j\varepsilon\frac{k}{n}
=
O\!\left(\frac{\log^2 x_n}{x_n}\right)=o(1),
\qquad
\frac{k}{n}
=
O\!\left(\frac{\varepsilon\log x_n}{x_n}\right)=o(1).
\]
Therefore $\log\frac{a'_{n,j}}{a_{n,j}}=o(1)$ uniformly, and hence $\frac{a'_{n,j}}{a_{n,j}}=1+o(1).$

\end{proof}

We are now ready to give a decomposition of the second moment of the tree component count.
\begin{lemma}\label{lem:second-moment-window}
Fix $B<\infty$, and let $I_n\subseteq [2,B\varepsilon^{-2}\log x_n]\cap\mathbb N.$ Define $N_n(I_n)=\sum_{k\in I_n} C_k^{(n,m)}.$ Then
\[
\mathbb E\bigl[N_n(I_n)(N_n(I_n)-1)\bigr]
=
(1+o(1))\bigl(\mathbb E N_n(I_n)\bigr)^2.
\]
\end{lemma}

\begin{proof}
For $k\in I_n$ and $A\subseteq[n]$ with $|A|=k$, let $\mathbf 1_A:=\mathbf 1_{\{A\text{ is a tree component of }G_{n,m}^{\alpha,\ast}\}}.$ Let $N_{\mathrm{ext}}(A)$ denote the number of tree components of $G_{n,m}^{\alpha,\ast}$ whose sizes belong to
$I_n$ and which are disjoint from $A$. Since distinct tree components are disjoint,
\[
N_n(I_n)\bigl(N_n(I_n)-1\bigr)
=
\sum_{k\in I_n}\ \sum_{\substack{A\subseteq[n]\\ |A|=k}}
\mathbf 1_A\,N_{\mathrm{ext}}(A).
\]
Taking expectations gives
\[
\mathbb E\bigl[N_n(I_n)(N_n(I_n)-1)\bigr]
=
\sum_{k\in I_n}\ \sum_{\substack{A\subseteq[n]\\ |A|=k}}
\mathbb P(A\text{ is a tree component})\,
\mathbb E\bigl[N_{\mathrm{ext}}(A)\mid A\text{ is a tree component}\bigr].
\]

Now fix $k\in I_n$ and $A\subseteq[n]$ with $|A|=k$. By Proposition~\ref{prop:regeneration}, conditional on the event that $A$ is a tree component, the induced multigraph on \([n]\setminus A\), after relabeling
\([n]\setminus A\) as \([n-k]\), has the law of
\(G_{n-k,m-k+1}^{\alpha,*}\). Therefore
\[
\mathbb E\bigl[N_{\mathrm{ext}}(A)\mid A\text{ is a tree component}\bigr]
=
\sum_{j\in I_n}\mathbb E C_j^{(n-k,m-k+1)}.
\]
Note that $I_n\subseteq [2,B\varepsilon^{-2}\log x_n]$, Lemma~\ref{lem:compare-expect} yields, uniformly for $j,k\in I_n$,
\[
\mathbb E C_j^{(n-k,m-k+1)}=(1+o(1))\mathbb E C_j^{(n,m)}.
\]
Hence, uniformly in $k\in I_n$ and $A$ with $|A|=k$,
\begin{equation}
\label{eq:external-tree-count-expectation}
\mathbb{E}\bigl[N_{\mathrm{ext}}(A)\mid A\text{ is a tree component}\bigr]
=
(1+o(1))\sum_{j\in I_n}\mathbb{E} C_j^{(n,m)}
=
(1+o(1))\mathbb{E} N_n(I_n).
\end{equation}
Substituting \eqref{eq:external-tree-count-expectation} back yields
\[
\mathbb E\bigl[N_n(I_n)(N_n(I_n)-1)\bigr]
=
(1+o(1))
\sum_{k\in I_n}
\left(
\sum_{\substack{A\subseteq[n]\\ |A|=k}}
\mathbb P(A\text{ is a tree component})
\right)
\mathbb E N_n(I_n).
\]
The inner sum is $\mathbb E C_k^{(n,m)}$, so
\[
\mathbb E\bigl[N_n(I_n)(N_n(I_n)-1)\bigr]
=
(1+o(1))
\left(
\sum_{k\in I_n}\mathbb E C_k^{(n,m)}
\right)
\mathbb E N_n(I_n)
=
(1+o(1))\bigl(\mathbb E N_n(I_n)\bigr)^2.
\]
\end{proof}

\begin{theorem}\label{thm:lower-multi-unified}
Fix $\delta>0$. Then
\[
\mathbb P\bigl(L_1(G_{n,m}^{\alpha,\ast})<(1-\delta)c_n\bigr)\to0.
\]
\end{theorem}

\begin{proof}
If $\delta\ge1$, then $(1-\delta)c_n\le0$, and the claim is immediate. We may therefore assume that $0<\delta<1.$ Set $\eta=\delta/4.$ Let $K_n=\left\lfloor (1-\eta)\frac{\log x_n}{\Psi_n}\right\rfloor, W_n=\left\lfloor \Psi_n^{-1}\right\rfloor,$ and $$I_n=\{K_n,K_n+1,\dots,K_n+W_n\}.$$
Since
\[
\Psi_n=\frac{\alpha+1}{2(\alpha+2)}\varepsilon^2+O(\varepsilon^3)
=
C_\alpha^{-1}\varepsilon^2(1+o(1)),
\]
we have
\[
K_n=(1-\eta+o(1))\,C_\alpha\,\varepsilon^{-2}\log x_n.
\]

Since $1-\eta>1-\delta$, it follows that $K_n\ge (1-\delta)c_n$ for all sufficiently large $n$. Also, $C_\alpha\in[2,4]$, so there exists a fixed constant $B>4$ such that
\[
I_n\subseteq [2,B\varepsilon^{-2}\log x_n]\cap\mathbb N
\]
for all large $n$.
Define $N_n=N_n(I_n)=\sum_{k\in I_n} C_k^{(n,m)}.$ We first show that $\mathbb E N_n\to\infty$. By Lemma~\ref{lem:ank-unified}, uniformly for $k\in I_n$,
\[
\mathbb E C_k^{(n,m)}\ge c\, n k^{-5/2} e^{-\Psi_n k}
\]
for some constant $c>0$. Since $k\le K_n+W_n$, we have
\[
\Psi_n k\le \Psi_n(K_n+W_n)\le (1-\eta)\log x_n+1
\]
for all large $n$, and thus
\[
e^{-\Psi_n k}\ge e^{-1}x_n^{-(1-\eta)}.
\]
Moreover, recall that $\log x_n \to \infty$ and $W_n \le \Psi_n^{-1},$ for all large $n$,
\[
k\le K_n+W_n\le \frac{(1-\eta)\log x_n}{\Psi_n}+\frac1{\Psi_n}\le \frac{2\log x_n}{\Psi_n},
\]
Hence
\[
k^{-5/2}\ge 2^{-5/2}\Psi_n^{5/2}(\log x_n)^{-5/2}.
\]
Therefore, uniformly for $k\in I_n$,
\[
\mathbb E C_k^{(n,m)}
\ge c' n \Psi_n^{5/2}(\log x_n)^{-5/2}x_n^{-(1-\eta)}
\]
for some constant $c'>0$. Since $|I_n|=W_n+1\ge \frac{1}{2\Psi_n}$ for all large $n$, we obtain
\[
\mathbb E N_n
\ge c'' n \Psi_n^{3/2}(\log x_n)^{-5/2}x_n^{-(1-\eta)}
\]
for some constant $c''>0$.
Using $\Psi_n\asymp \varepsilon^2$, we obtain
\[
\mathbb E N_n
\ge c''' x_n^\eta (\log x_n)^{-5/2}\to\infty.
\]

Next, Lemma~\ref{lem:second-moment-window} yields
\[
\mathbb E\bigl[N_n(N_n-1)\bigr]=(1+o(1))(\mathbb E N_n)^2.
\]
Since $\mathbb E N_n\to\infty$, we also have
\[
\mathbb E N_n^2
=
\mathbb E\bigl[N_n(N_n-1)\bigr]+\mathbb E N_n
=
(1+o(1))(\mathbb E N_n)^2.
\]
By the Paley--Zygmund inequality,
\[
\mathbb P(N_n>0)\ge \frac{(\mathbb E N_n)^2}{\mathbb E N_n^2}=1-o(1).
\]
On the event $\{N_n>0\}$ there exists a tree component with size in $I_n$, so
\[
L_1(G_{n,m}^{\alpha,\ast})\ge \min I_n=K_n\ge (1-\delta)c_n
\]
for all large $n$. Therefore
\[
\mathbb P\bigl(L_1(G_{n,m}^{\alpha,\ast})<(1-\delta)c_n\bigr)\to0,
\]
as claimed.
\end{proof}

\begin{theorem}[Lower bound for the simple graph]\label{thm:lower-simple-unified}
Fix $\delta>0$, then
\[
\mathbb P\bigl(L_1(G_{n,m}^{\alpha})<(1-\delta)c_n\bigr)\to0.
\]
\end{theorem}

\begin{proof}
If $\delta\ge1$, then $(1-\delta)c_n\le0$, so the claim is immediate. Assume that $0<\delta<1$.
Let $\mathcal G_n^-$ be the family of simple graphs on $[n]$ with exactly $m$ edges and largest component
size less than $(1-\delta)c_n$, namely
\[
\mathcal G_n^-=\{G:\ G\text{ is simple graph on }[n],\ e(G)=m,\ L_1(G)<(1-\delta)c_n\}.
\]
Since
\[
m=\frac{\alpha}{2(\alpha+1)}(1-\varepsilon)n\le \frac n2
\]
and $\alpha\ge \alpha_*$ for all large $n$, Lemma~\ref{lem:jw-transfer} applies with $C=1$ and $\alpha_0=\alpha_*$.
Thus there exists a constant $B_0=B_0(1,\alpha_*)<\infty$ such that
\[
\mathbb P(G_{n,m}^{\alpha}\in\mathcal G_n^-)
\le
B_0\,\mathbb P(G_{n,m}^{\alpha,\ast}\in\mathcal G_n^-)+o(1).
\]
Since every graph in \(\mathcal G_n^-\) is simple,
\[
\mathbb P(G_{n,m}^{\alpha,\ast}\in\mathcal G_n^-)
\le
\mathbb P\bigl(L_1(G_{n,m}^{\alpha,\ast})<(1-\delta)c_n\bigr)\to0
\]
by Theorem~\ref{thm:lower-multi-unified}. Therefore
\[
\mathbb P\bigl(L_1(G_{n,m}^{\alpha})<(1-\delta)c_n\bigr)
=
\mathbb P(G_{n,m}^{\alpha}\in\mathcal G_n^-)
\to0.
\]

\end{proof}

Recall $c_n=C_\alpha\,\varepsilon^{-2}\log(\varepsilon^3 n),$ and fix $\delta>0$. Theorem~\ref{thm:upper-unified} gives
\[
\mathbb P\bigl(L_1(G_{n,m}^{\alpha,\ast})>(1+\delta)c_n\bigr)\to0,
\qquad
\mathbb P\bigl(L_1(G_{n,m}^{\alpha})>(1+\delta)c_n\bigr)\to0.
\]
Theorems~\ref{thm:lower-multi-unified} and~\ref{thm:lower-simple-unified} give
\[
\mathbb P\bigl(L_1(G_{n,m}^{\alpha,\ast})<(1-\delta)c_n\bigr)\to0,
\qquad
\mathbb P\bigl(L_1(G_{n,m}^{\alpha})<(1-\delta)c_n\bigr)\to0.
\]
Therefore
\[
\mathbb P\bigl((1-\delta)c_n\le L_1(G_{n,m}^{\alpha,\ast})\le (1+\delta)c_n\bigr)\to1,
\]
and
\[
\mathbb P\bigl((1-\delta)c_n\le L_1(G_{n,m}^{\alpha})\le (1+\delta)c_n\bigr)\to1.
\]
Since $\delta>0$ is arbitrary, these are equivalent to
\[
L_1(G_{n,m}^{\alpha,\ast})
=
(1+o_p(1))\,c_n,
\qquad
L_1(G_{n,m}^{\alpha})
=
(1+o_p(1))\,c_n.
\]
\begin{remark}
The assumption $\varepsilon=o(1)$ in Theorem~\ref{subcritical case theo 2}
is tied to the barely subcritical normalization.  If $\varepsilon\in(0,1)$ is fixed, the same moment computation gives
\[
        L_1(G^\alpha_{n,m})
        =(1+o_p(1))\,\frac{\log n}{\Psi_\alpha(\varepsilon)},
\]
 where
\[
        \Psi_\alpha(\varepsilon)
        =
        \log\frac{1}{1-\varepsilon}
        +(\alpha+2)\log\left(1-\frac{\varepsilon}{\alpha+2}\right).
\]

The reason is that, for fixed $\varepsilon$, the relevant exponential decay is governed by the full
Cram\'er exponent rather than by $\varepsilon^{-2}$ at the critical
point.  In the multigraph model, the limiting degree distribution is
$D\sim \operatorname{NBin}(\alpha,p)$ with $ p=\frac{1-\varepsilon}{\alpha+2-\varepsilon}.$ Then
$D^\star-1\sim \operatorname{NBin}(\alpha+1,p)$.  Therefore, for
$\eta=D^\star-2$,
\[
        \varphi_\alpha(\theta)
        =\mathbb E e^{\theta\eta}
        =
        e^{-\theta}
        \left(\frac{1-p}{1-pe^\theta}\right)^{\alpha+1}.
\]
The minimum of $\varphi_\alpha$ over positive $\theta$ is attained at
$\theta=\theta_\alpha$ satisfying
\begin{equation}
    pe^{\theta_\alpha}=\frac{1}{\alpha+2}.
    \label{eq:theta_alpha}
\end{equation}
Substituting \eqref{eq:theta_alpha} gives
\begin{equation}
    -\log \varphi_\alpha(\theta_\alpha)
    =
    \log\frac{1}{1-\varepsilon}
    +(\alpha+2)\log\left(1-\frac{\varepsilon}{\alpha+2}\right)
    =
    \Psi_\alpha(\varepsilon).
    \label{eq:log_varphi}
\end{equation}
Thus the fixed subcritical scale is $\Psi_\alpha(\varepsilon)^{-1}\log n$.
The constant appearing in Theorem~\ref{subcritical case theo 2} is 
the near critical expansion of \eqref{eq:log_varphi}, since
\[
        \Psi_\alpha(\varepsilon)
        =
        \frac{\alpha+1}{2(\alpha+2)}\varepsilon^2
        +O(\varepsilon^3),
        \qquad \varepsilon\downarrow0.
\]

The lower bound requires no new idea, because it does not depend on \(\varepsilon = o(1)\). For the upper bound,  
for a good deterministic degree sequence, the component exploration in the
configuration model admits the exponential bound
\[
        \mathbb P_d\big(|\mathcal C(v)|\ge k\big)
        \le
        \exp\{O(d_v)\}
        \exp\{-(\Psi_\alpha(\varepsilon)-o(1))k\}.
\]
The required exponential moment of the degree sequence holds because
$pe^{\theta_\alpha}=1/(\alpha+2)<1$.  Summing over all vertices and
choosing $k=(1+\delta)\Psi_\alpha(\varepsilon)^{-1}\log n$ gives the upper bound for the multigraph model.  The transfer from
$G^{\alpha,*}_{n,m}$ to the simple graph process $G^\alpha_{n,m}$ then follows
as in the proof of Theorem~\ref{subcritical case theo 2}.

\end{remark}

The result about $\alpha \to\infty $ follows directly.

\begin{corollary}\label{cor:subcritical-jw}
Suppose that $\alpha\to\infty,\alpha\varepsilon\to\infty,
\varepsilon^3 n\to\infty.$ Let $m=\left\lfloor \frac n2(1-\varepsilon)\right\rfloor.$ Then
\[
L_1(G_{n,m}^{\alpha})=(2+o_p(1))\varepsilon^{-2}\log(\varepsilon^3 n).
\]
\end{corollary}

\begin{proof}
Write $m=\frac n2(1-\varepsilon)-\zeta_n, 0\le \zeta_n<1,$ and define $\tilde\varepsilon=1-\frac{2(\alpha+1)m}{\alpha n}.$ Then $m=\frac{\alpha n}{2(\alpha+1)}(1-\tilde\varepsilon),$ so Theorem \ref{thm:main-fixed-alpha} applies with $\tilde\varepsilon$ in place of $\varepsilon$. Therefore,
\[
\tilde\varepsilon
=1-\frac{\alpha+1}{\alpha}(1-\varepsilon)+\frac{2(\alpha+1)\zeta_n}{\alpha n}
=\varepsilon-\frac{1-\varepsilon}{\alpha}+O(n^{-1}).
\]
Since $\alpha\varepsilon\to\infty$, we have $\alpha^{-1}=o(\varepsilon)$. Since $\varepsilon^3 n\to\infty$, we also have $n^{-1}=o(\varepsilon)$. Hence $\tilde\varepsilon=\varepsilon(1+o(1)).$ In particular, $\tilde\varepsilon\to0$, $\tilde\varepsilon>0$ for all large $n$, and
\[
\tilde\varepsilon^3 n=(1+o(1))\varepsilon^3 n\to\infty.
\]

Therefore Theorem \ref{thm:main-fixed-alpha} yields
\begin{equation}
\label{eq:L1-tilde}
L_1(G_{n,m}^{\alpha})
=(1+o_p(1))\frac{2(\alpha+2)}{\alpha+1}
\tilde\varepsilon^{-2}\log(\tilde\varepsilon^3 n).
\end{equation}
Now $\alpha\to\infty$ implies
\begin{equation}
\label{eq:alpha-factor-and-eps}
\frac{2(\alpha+2)}{\alpha+1}=2+o(1),
\qquad
\tilde\varepsilon^{-2}=(1+o(1))\varepsilon^{-2}.
\end{equation}

Also, since $\tilde\varepsilon/\varepsilon\to1$ and $\varepsilon^3 n\to\infty$,
\begin{equation}
\label{eq:log-tilde-eps}
\begin{aligned}
\log(\tilde\varepsilon^3 n)
&=\log(\varepsilon^3 n)
  +3\log\!\left(\frac{\tilde\varepsilon}{\varepsilon}\right) \\
&=\log(\varepsilon^3 n)+o(1) \\
&=(1+o(1))\log(\varepsilon^3 n).
\end{aligned}
\end{equation}
Combining \eqref{eq:L1-tilde}-\eqref{eq:log-tilde-eps}, we obtain
\[
L_1(G_{n,m}^{\alpha})=(2+o_p(1))\varepsilon^{-2}\log(\varepsilon^3 n).
\]

\end{proof}

\section*{Further Work}
%
%
%

We conclude with several directions for further research.
The upper bound argument developed in this paper combines the construction of
a good degree sequence with an upper bound for the subcritical configuration
model, while the lower bound argument is based on a second moment analysis of
tree components.

A similar strategy may also provide a route to
Problem~5.4 of Janson and Warnke~\cite{JansonWarnke2021Preferential}.  For the upper bound, one needs to show that the
random degree sequence of the graph remaining after the removal of the giant
component belongs, with high probability, to a deterministic class of
good degree sequences.

In addition, the first and second moment estimates for tree components
used in our lower bound argument appear to extend to the supercritical
regime, suggesting that, in the multigraph model, the second largest component should satisfy the sharp asymptotic 
\[
L_2 \sim C_\alpha \varepsilon^{-2}\log(\varepsilon^3 n).
\]

 It would also be interesting to
investigate whether analogous ideas can be extended to vertex heterogeneous
affine weights $(d_u+\alpha_u)(d_v+\alpha_v)$ and to hypergraph variants.

\begin{acks}[Acknowledgments]
The author would like to thank Shuyang Gong and Lutz Warnke for their valuable suggestions.
\end{acks}

\begin{funding}
 The author is supported by the National Natural Science Foundation of China
(12595294, 12231002), and the New Cornerstone Science Foundation(NCI202501).
\end{funding}



\begin{thebibliography}{99}

\bibitem{Aldous1997Brownian}
\begin{barticle}[author]
\bauthor{\bsnm{Aldous},~\binits{D.}}
(\byear{1997}).
\btitle{Brownian excursions, critical random graphs and the multiplicative coalescent}.
\bjournal{Ann. Probab.} \bvolume{25} \bpages{812--854}.
\bid{doi={10.1214/aop/1024404421}}
\end{barticle}
\endbibitem

\bibitem{Barabasi2016NetworkScience}
\begin{bbook}[author]
\bauthor{\bsnm{Barab\'asi},~\binits{A.-L.}}
(\byear{2016}).
\btitle{Network Science}.
\bpublisher{Cambridge University Press}, \blocation{Cambridge}.
\end{bbook}
\endbibitem

\bibitem{BarabasiAlbert1999Emergence}
\begin{barticle}[author]
\bauthor{\bsnm{Barab\'asi},~\binits{A.-L.}} \AND
\bauthor{\bsnm{Albert},~\binits{R.}}
(\byear{1999}).
\btitle{Emergence of scaling in random networks}.
\bjournal{Science} \bvolume{286} \bpages{509--512}.
\bid{doi={10.1126/science.286.5439.509}}
\end{barticle}
\endbibitem

\bibitem{BenNaimKrapivsky2012Popularity}
\begin{barticle}[author]
\bauthor{\bsnm{Ben-Naim},~\binits{E.}} \AND
\bauthor{\bsnm{Krapivsky},~\binits{P. L.}}
(\byear{2012}).
\btitle{Popularity-driven networking}.
\bjournal{EPL} \bvolume{97} \bpages{48003}.
\bid{doi={10.1209/0295-5075/97/48003}}
\end{barticle}
\endbibitem

\bibitem{Bollobas2001RandomGraphs}
\begin{bbook}[author]
\bauthor{\bsnm{Bollob\'as},~\binits{B.}}
(\byear{2001}).
\btitle{Random Graphs}, 2nd ed.
\bpublisher{Cambridge University Press}, \blocation{Cambridge}.
\bid{doi={10.1017/CBO9780511814068}}
\end{bbook}
\endbibitem

\bibitem{BollobasRiordan2013PhaseTransition}
\begin{bincollection}[author]
\bauthor{\bsnm{Bollob\'as},~\binits{B.}} \AND
\bauthor{\bsnm{Riordan},~\binits{O.}}
(\byear{2013}).
\btitle{The phase transition in the Erd\H{o}s--R\'enyi random graph process}.
In \bbooktitle{Erd\H{o}s Centennial} \bpages{59--110}.
\bpublisher{J\'anos Bolyai Mathematical Society}, \blocation{Budapest}.
\bid{doi={10.1007/978-3-642-39286-3_3}}
\end{bincollection}
\endbibitem

\bibitem{BollobasRiordanSpencerTusnady2001DegreeSequence}
\begin{barticle}[author]
\bauthor{\bsnm{Bollob\'as},~\binits{B.}},
\bauthor{\bsnm{Riordan},~\binits{O.}},
\bauthor{\bsnm{Spencer},~\binits{J.}} \AND
\bauthor{\bsnm{Tusn\'ady},~\binits{G.}}
(\byear{2001}).
\btitle{The degree sequence of a scale-free random graph process}.
\bjournal{Random Structures Algorithms} \bvolume{18} \bpages{279--290}.
\bid{doi={10.1002/rsa.1009}}
\end{barticle}
\endbibitem

\bibitem{BorgsChayesLovaszSosVesztergombi2011Limits}
\begin{barticle}[author]
\bauthor{\bsnm{Borgs},~\binits{C.}},
\bauthor{\bsnm{Chayes},~\binits{J.}},
\bauthor{\bsnm{Lov\'asz},~\binits{L.}},
\bauthor{\bsnm{S\'os},~\binits{V. T.}} \AND
\bauthor{\bsnm{Vesztergombi},~\binits{K.}}
(\byear{2011}).
\btitle{Limits of randomly grown graph sequences}.
\bjournal{European J. Combin.} \bvolume{32} \bpages{985--999}.
\bid{doi={10.1016/j.ejc.2011.03.015}}
\end{barticle}
\endbibitem

\bibitem{CoulsonPerarnau2023Subcritical}
\begin{barticle}[author]
\bauthor{\bsnm{Coulson},~\binits{M.}} \AND
\bauthor{\bsnm{Perarnau},~\binits{G.}}
(\byear{2023}).
\btitle{Largest component of subcritical random graphs with given degree sequence}.
\bjournal{Electron. J. Probab.} \bvolume{28} \bpages{Paper No. 34, 28 pp}.
\bid{doi={10.1214/23-EJP921}}
\end{barticle}
\endbibitem

\bibitem{DharaHofstadLeeuwaardenSen2017CriticalWindow}
\begin{barticle}[author]
\bauthor{\bsnm{Dhara},~\binits{S.}},
\bauthor{\bparticle{van der}~\bsnm{Hofstad},~\binits{R.}},
\bauthor{\bparticle{van}~\bsnm{Leeuwaarden},~\binits{J. S. H.}} \AND
\bauthor{\bsnm{Sen},~\binits{S.}}
(\byear{2017}).
\btitle{Critical window for the configuration model: finite third moment degrees}.
\bjournal{Electron. J. Probab.} \bvolume{22} \bpages{Paper No. 16, 33 pp}.
\bid{doi={10.1214/17-EJP29}}
\end{barticle}
\endbibitem

\bibitem{DingKimLubetzkyPeres2011YoungGiant}
\begin{barticle}[author]
\bauthor{\bsnm{Ding},~\binits{J.}},
\bauthor{\bsnm{Kim},~\binits{J. H.}},
\bauthor{\bsnm{Lubetzky},~\binits{E.}} \AND
\bauthor{\bsnm{Peres},~\binits{Y.}}
(\byear{2011}).
\btitle{Anatomy of a young giant component in the random graph}.
\bjournal{Random Structures Algorithms} \bvolume{39} \bpages{139--178}.
\bid{doi={10.1002/rsa.20342}}
\end{barticle}
\endbibitem

\bibitem{DingLubetzkyPeres2014StrictlySupercritical}
\begin{barticle}[author]
\bauthor{\bsnm{Ding},~\binits{J.}},
\bauthor{\bsnm{Lubetzky},~\binits{E.}} \AND
\bauthor{\bsnm{Peres},~\binits{Y.}}
(\byear{2014}).
\btitle{Anatomy of the giant component: The strictly supercritical regime}.
\bjournal{European J. Combin.} \bvolume{35} \bpages{155--168}.
\bid{doi={10.1016/j.ejc.2013.06.004}}
\end{barticle}
\endbibitem

\bibitem{ErdosRenyi1959Random}
\begin{barticle}[author]
\bauthor{\bsnm{Erd\H{o}s},~\binits{P.}} \AND
\bauthor{\bsnm{R\'enyi},~\binits{A.}}
(\byear{1959}).
\btitle{On random graphs. I}.
\bjournal{Publ. Math. Debrecen} \bvolume{6} \bpages{290--297}.
\end{barticle}
\endbibitem

\bibitem{ErdosRenyi1960Evolution}
\begin{barticle}[author]
\bauthor{\bsnm{Erd\H{o}s},~\binits{P.}} \AND
\bauthor{\bsnm{R\'enyi},~\binits{A.}}
(\byear{1960}).
\btitle{On the evolution of random graphs}.
\bjournal{Publ. Math. Inst. Hungar. Acad. Sci.} \bvolume{5} \bpages{17--61}.
\end{barticle}
\endbibitem

\bibitem{HatamiMolloy2012ScalingWindow}
\begin{barticle}[author]
\bauthor{\bsnm{Hatami},~\binits{H.}} \AND
\bauthor{\bsnm{Molloy},~\binits{M.}}
(\byear{2012}).
\btitle{The scaling window for a random graph with a given degree sequence}.
\bjournal{Random Structures Algorithms} \bvolume{41} \bpages{99--123}.
\bid{doi={10.1002/RSA.20394}}
\end{barticle}
\endbibitem

\bibitem{HruzPeter2010Nongrowing}
\begin{barticle}[author]
\bauthor{\bsnm{Hru\v{z}},~\binits{T.}} \AND
\bauthor{\bsnm{Peter},~\binits{U.}}
(\byear{2010}).
\btitle{Nongrowing preferential attachment random graphs}.
\bjournal{Internet Math.} \bvolume{6} \bpages{461--487}.
\bid{doi={10.1080/15427951.2010.553143}}
\end{barticle}
\endbibitem

\bibitem{Janson2018EdgeExchangeable}
\begin{barticle}[author]
\bauthor{\bsnm{Janson},~\binits{S.}}
(\byear{2018}).
\btitle{On edge exchangeable random graphs}.
\bjournal{J. Stat. Phys.} \bvolume{173} \bpages{448--484}.
\bid{doi={10.1007/s10955-017-1832-9}}
\end{barticle}
\endbibitem

\bibitem{JansonLuczak2009NewApproach}
\begin{barticle}[author]
\bauthor{\bsnm{Janson},~\binits{S.}} \AND
\bauthor{\bsnm{{\L}uczak},~\binits{M. J.}}
(\byear{2009}).
\btitle{A new approach to the giant component problem}.
\bjournal{Random Structures Algorithms} \bvolume{34} \bpages{197--216}.
\bid{doi={10.1002/rsa.20231}}
\end{barticle}
\endbibitem

\bibitem{JansonLuczakRucinski2000RandomGraphs}
\begin{bbook}[author]
\bauthor{\bsnm{Janson},~\binits{S.}},
\bauthor{\bsnm{{\L}uczak},~\binits{T.}} \AND
\bauthor{\bsnm{Ruci\'nski},~\binits{A.}}
(\byear{2000}).
\btitle{Random Graphs}.
\bpublisher{Wiley}, \blocation{New York}.
\bid{doi={10.1002/9781118032718}}
\end{bbook}
\endbibitem

\bibitem{JansonWarnke2021Preferential}
\begin{barticle}[author]
\bauthor{\bsnm{Janson},~\binits{S.}} \AND
\bauthor{\bsnm{Warnke},~\binits{L.}}
(\byear{2021}).
\btitle{Preferential attachment without vertex growth: emergence of the giant component}.
\bjournal{Ann. Appl. Probab.} \bvolume{31} \bpages{1523--1547}.
\bid{doi={10.1214/20-AAP1610}}
\end{barticle}
\endbibitem

\bibitem{Joseph2014CriticalDegrees}
\begin{barticle}[author]
\bauthor{\bsnm{Joseph},~\binits{A.}}
(\byear{2014}).
\btitle{The component sizes of a critical random graph with given degree sequence}.
\bjournal{Ann. Appl. Probab.} \bvolume{24} \bpages{2560--2594}.
\bid{doi={10.1214/13-AAP985}}
\end{barticle}
\endbibitem

\bibitem{KangSeierstad2008CriticalPhase}
\begin{barticle}[author]
\bauthor{\bsnm{Kang},~\binits{M.}} \AND
\bauthor{\bsnm{Seierstad},~\binits{T. G.}}
(\byear{2008}).
\btitle{The critical phase for random graphs with a given degree sequence}.
\bjournal{Combin. Probab. Comput.} \bvolume{17} \bpages{67--86}.
\bid{doi={10.1017/S096354830700867X}}
\end{barticle}
\endbibitem

\bibitem{MolloyReed1995CriticalPoint}
\begin{barticle}[author]
\bauthor{\bsnm{Molloy},~\binits{M.}} \AND
\bauthor{\bsnm{Reed},~\binits{B.}}
(\byear{1995}).
\btitle{A critical point for random graphs with a given degree sequence}.
\bjournal{Random Structures Algorithms} \bvolume{6} \bpages{161--180}.
\bid{doi={10.1002/rsa.3240060204}}
\end{barticle}
\endbibitem

\bibitem{MolloyReed1998SizeGiant}
\begin{barticle}[author]
\bauthor{\bsnm{Molloy},~\binits{M.}} \AND
\bauthor{\bsnm{Reed},~\binits{B.}}
(\byear{1998}).
\btitle{The size of the giant component of a random graph with a given degree sequence}.
\bjournal{Combin. Probab. Comput.} \bvolume{7} \bpages{295--305}.
\bid{doi={10.1017/S0963548398003526}}
\end{barticle}
\endbibitem

\bibitem{Pittel2010Degrees}
\begin{barticle}[author]
\bauthor{\bsnm{Pittel},~\binits{B.}}
(\byear{2010}).
\btitle{On a random graph evolving by degrees}.
\bjournal{Adv. Math.} \bvolume{223} \bpages{619--671}.
\bid{doi={10.1016/j.aim.2009.08.015}}
\end{barticle}
\endbibitem

\bibitem{Price1976Cumulative}
\begin{barticle}[author]
\bauthor{\bparticle{de Solla}~\bsnm{Price},~\binits{D. J.}}
(\byear{1976}).
\btitle{A general theory of bibliometric and other cumulative advantage processes}.
\bjournal{J. Amer. Soc. Inform. Sci.} \bvolume{27} \bpages{292--306}.
\bid{doi={10.1002/ASI.4630270505}}
\end{barticle}
\endbibitem

\bibitem{Rath2012EdgeConservative}
\begin{barticle}[author]
\bauthor{\bsnm{R\'ath},~\binits{B.}}
(\byear{2012}).
\btitle{Time evolution of dense multigraph limits under edge-conservative preferential attachment dynamics}.
\bjournal{Random Structures Algorithms} \bvolume{41} \bpages{365--390}.
\bid{doi={10.1002/rsa.20422}}
\end{barticle}
\endbibitem

\bibitem{RathSzakacs2012Multigraph}
\begin{barticle}[author]
\bauthor{\bsnm{R\'ath},~\binits{B.}} \AND
\bauthor{\bsnm{Szak\'acs},~\binits{L.}}
(\byear{2012}).
\btitle{Multigraph limit of the dense configuration model and the preferential attachment graph}.
\bjournal{Acta Math. Hungar.} \bvolume{136} \bpages{196--221}.
\bid{doi={10.1007/s10474-012-0217-4}}
\end{barticle}
\endbibitem

\bibitem{vanDerHofstad2017RGCN1}
\begin{bbook}[author]
\bauthor{\bparticle{van der}~\bsnm{Hofstad},~\binits{R.}}
(\byear{2017}).
\btitle{Random Graphs and Complex Networks. Volume 1}.
\bpublisher{Cambridge University Press}, \blocation{Cambridge}.
\bid{doi={10.1017/9781316779422}}
\end{bbook}
\endbibitem

\bibitem{vanDerHofstad2024RGCN2}
\begin{bbook}[author]
\bauthor{\bparticle{van der}~\bsnm{Hofstad},~\binits{R.}}
(\byear{2024}).
\btitle{Random Graphs and Complex Networks. Volume 2}.
\bpublisher{Cambridge University Press}, \blocation{Cambridge}.
\bid{doi={10.1017/9781316795552}}
\end{bbook}
\endbibitem

\bibitem{vanDerHofstadJansonLuczak2019BarelySupercritical}
\begin{barticle}[author]
\bauthor{\bparticle{van der}~\bsnm{Hofstad},~\binits{R.}},
\bauthor{\bsnm{Janson},~\binits{S.}} \AND
\bauthor{\bsnm{{\L}uczak},~\binits{M. J.}}
(\byear{2019}).
\btitle{Component structure of the configuration model: barely supercritical case}.
\bjournal{Random Structures Algorithms} \bvolume{55} \bpages{3--55}.
\bid{doi={10.1002/rsa.20837}}
\end{barticle}
\endbibitem

\end{thebibliography}
\end{document}